\documentclass[11pt,reqno]{amsart}
\usepackage{amsmath,amssymb,amsthm,mathtools,calc,verbatim,enumitem,tikz,url,hyperref,mathrsfs,cite,fullpage}
\usepackage{bbm}
\usepackage{pgfplots}
 \pgfplotsset{compat=1.18}
 \usepackage{tikz}
\usepackage{pgfplots}
 \pgfplotsset{compat=1.18}
 \usetikzlibrary{arrows.meta}
\usepackage{todonotes}
\usepackage{stmaryrd}
\usepackage{textcomp}
\usepackage{setspace}
\usepackage{amssymb}
\usepackage{amsthm}
\usepackage{amsmath}
\usepackage{graphicx}
\usetikzlibrary{calc}
\usepackage{marvosym}
\usepackage{empheq}
\usepackage{latexsym}
\usepackage{fontenc}
\usepackage{color}
\usepackage{hyperref}
\usepackage{cleveref}
\usepackage{dsfont}
\usepackage{bbm}

\usepackage[T1]{fontenc}
\usepackage[utf8]{inputenc}

\usepackage{booktabs}
\usepackage{array}

\usepackage{todonotes}

\newtheorem{theorem}{Theorem}[section]
\newtheorem{lemma}[theorem]{Lemma}
\newtheorem{corollary}[theorem]{Corollary}

\newtheorem{observation}[theorem]{Observation}
\newtheorem{proposition}[theorem]{Proposition}

\newtheorem*{claim*}{Claim}
\newtheorem{problem}[theorem]{Problem}
	
\theoremstyle{definition}
\newtheorem{definition}[theorem]{Definition}

\newtheorem*{qu*}{Question}
\theoremstyle{remark}

\newcommand\cB{\mathcal{B}}
\newcommand\cC{\mathcal{C}}
\newcommand\cD{\mathcal{D}}
\newcommand\cF{\mathcal{F}}
\newcommand\cG{\mathcal{G}}
\newcommand\cH{\mathcal{H}}

\newcommand\cK{\mathcal{K}}

\newcommand\cM{\mathcal{M}}

\newcommand\cS{\mathcal{S}}
\newcommand\cX{\mathcal{X}}

\newcommand\cR{\mathcal{R}}

\newcommand\cV{\mathcal{V}}

\renewcommand\leq{\leqslant}
\renewcommand\geq{\geqslant}
\renewcommand\le{\leqslant}
\renewcommand\ge{\geqslant}
\renewcommand\to{\rightarrow}

	\def\t{\theta}

	\def\cF{\mathcal{F}}

	\def\<{\langle }
	\def\>{\rangle }

\newcounter{paperpart}

\begin{document}

\title{Structural Reductions for Monochromatic Matchings and Ramsey Tilings}
\author{Hong Liu \and Maksim Turevskii \and Lanchao Wang \and Zhifei Yan}

\address{ECOPRO, Institute for Basic Science, 55 Expo-ro, Yuseong-gu, Daejeon, 34126, Korea}\email{\{hongliu,zhifeiyan\}@ibs.re.kr} 
\address{Faculty of Mathematics and Computer Science, Saint Petersburg State University, Saint Petersburg, 199178, Russia}\email{turmax20052005@gmail.com} 

\address{School of Mathematics, Nanjing University, Nanjing, China, and ECOPRO, Institute for Basic Science, 55 Expo-ro, Yuseong-gu, Daejeon, 34126, Korea}\email{lanchaowang@foxmail.com}

\begin{abstract}
The Alon--Frankl--Lov\'asz theorem determines the chromatic number of Kneser
hypergraphs; equivalently, it gives the sharp minimum size of a monochromatic
matching in every edge-colouring of a complete uniform hypergraph. Its known
general proofs are topological. We introduce a topology-free structural
framework. It reduces every colouring of a
pseudorandom $t$-graph, with only $o(n)$ loss in the largest monochromatic
matching, to a colouring of $K_n^{(t)}$ whose vertex set has at most $r$ parts
and whose edge colours depend only on intersection profiles. 
Together with a stability analysis at the critical scale, we prove an exact
robust form: there exists $c=c(r,t)>0$ such that, if a $t$-graph satisfies
$\delta(\cG)\ge(1-c)\binom{n-1}{t-1}$, then every $r$-colouring
of $\cG$ contains a monochromatic matching of the exact optimal size for all
sufficiently large $n$. This gives a topology-free proof of the
Alon--Frankl--Lov\'asz theorem for large $n$, a sparse random transference
theorem, and the exact value predicted by Meunier's stable Kneser conjecture
throughout the range covered by our robust AFL theorem.

We further develop the framework for Ramsey graph tilings.
For a graph $H$, let $Rt_r(H;K_n)$ be the minimum, over all
$r$-edge-colourings of $K_n$, of the largest monochromatic $H$-tiling. We
prove
\[
Rt_r(H;K_n)=(\beta_{r,H}+o(1))n,
\]
where $\beta_{r,H}$ is effectively computable from finitely many rational
linear programs depending only on $H$ and $r$. An additional multipartite Ramsey argument is needed to reconstruct a consistent
coloured template. This gives an effective asymptotic solution to the
multicolour Ramsey-tiling problem, extending the classical
two-colour theorem of Burr, Erd\H{o}s and Spencer. We also determine explicit constants for several natural families.
\end{abstract}

\maketitle


\section{Introduction}

A central problem in extremal combinatorics and Ramsey theory is to determine
how much monochromatic structure must occur in an arbitrary colouring. In
this paper we study this question for matchings in uniform hypergraphs and,
more generally, for tilings in graphs. Our starting point is the following
theorem of Alon, Frankl and Lov\'asz~\cite{AFL}, which resolved a conjecture
of Erd\H{o}s~\cite{Erdos} and generalized Kneser's original
conjecture~\cite{Kneser,Lovasz}. We state it in the equivalent language of
edge-coloured complete hypergraphs.

\begin{theorem}[Alon, Frankl and Lov\'asz]\label{thm:AFL}
Let \(r,t\ge2\) and \(n\in\mathbb N\), and let \(\chi\) be an arbitrary
\(r\)-edge-colouring of \(K_n^{(t)}\). Then \(\chi\) contains a
monochromatic matching consisting of
\[
\left\lfloor\frac{n+r-1}{t+r-1}\right\rfloor
\]
vertex-disjoint \(t\)-edges. Moreover, this bound is tight.
\end{theorem}

Theorem~\ref{thm:AFL} is a foundational result connecting Ramsey theory,
Kneser-type colouring problems, and topological combinatorics. Known general
proofs of the exact theorem use equivariant topology, through the
Borsuk--Ulam theorem, Tucker-type lemmas, or related topological
obstructions; see, for example,
\cite{AFL,Kriz,Matousek,Ziegler2002,Haviv}. A long-standing challenge is to
find a combinatorial mechanism behind the theorem: one which not only
recovers the numerical bound, but also exposes the structure responsible for
the critical denominator \(t+r-1\) and remains meaningful beyond complete
hypergraphs.

Our main result gives such a mechanism for fixed \(r,t\) and sufficiently
large \(n\). In fact, we prove a locally robust form of the exact theorem.
For a $t$-graph $\mathcal G$, let $\delta(\mathcal G)$ denote its minimum degree. If $\chi$ is an edge-colouring, let $\nu(\chi)$ denote the
maximum size of a monochromatic matching.

\begin{theorem}[Robust AFL theorem]
\label{thm:robust-critical-AFL}
Fix $r\ge1$ and $t\ge2$. There exist $c>0$ and $n_0$ such that,
for every $n\ge n_0$, every $t$-graph $\mathcal G$ on $n$ vertices
satisfying
\[
\delta(\mathcal G)\ge
(1-c)\binom{n-1}{t-1}
\]
has the following property: every $r$-edge-colouring $\chi$ of
$\mathcal G$ satisfies
\[
\nu(\chi)\ge
\left\lfloor\frac{n+r-1}{t+r-1}\right\rfloor.
\]
\end{theorem}

Thus, the AFL bound remains valid for sufficiently dense $t$-graphs with
minimum degree close to that of the complete $t$-graph. In
particular, taking $\mathcal G=K_n^{(t)}$ gives a topology-free proof of
the exact AFL formula for all sufficiently large $n$.

The proof has two distinct layers. The first is an asymptotic structural
reduction. An \emph{\(r\)-strip colouring} is an edge-colouring for which the vertex
set is partitioned into at most \(r\) parts and the colour of an edge depends
only on its intersection profile with these parts. Thus, once the part sizes
are specified, the colouring is encoded by finitely many profile colours.

\begin{theorem}[Structural reduction via strip colourings, informal]
\label{thm:informal-hypergraph-structure}
Let \(r,t\ge2\) be fixed, and let \(\mathcal G\) be a sufficiently large
pseudorandom \(t\)-graph on \(n\) vertices. For every \(r\)-edge-colouring
\(\chi\) of \(\mathcal G\), there exists an \(r\)-strip
\(r\)-edge-colouring \(\psi\) of \(K_n^{(t)}\) such that
\[
\nu(\chi)\ge\nu(\psi)-o(n).
\]
\end{theorem}

The formal statement is Theorem~\ref{thm:hypergraphstructure}. Weak
hypergraph regularity first reduces the original colouring to a bounded
system of clusters. For each colour, the fractional matching problem in the
reduced hypergraph has a dual fractional-cover problem, assigning one weight
to every cluster. Grouping these weights over all \(r\) colours gives vectors
\(S_i\in\mathbb R^r\): the \(\ell_\infty\)-norm of their sum controls the
matching bound, while the dual constraints say that almost every reduced
\(t\)-set is covered in some coordinate.

Two obstacles remain. First, the covering condition fails on the exceptional
tuples produced by regularity. We repair these exceptions by rounding the
vectors and replacing rare rounded values by frequent ones. Second, even
after repair there may be many vector types. We move the average of the
repaired vectors in the negative all-ones direction until it reaches the
boundary of their convex hull. A Carath\'eodory argument then represents
this boundary point using at most \(r\) frequent vectors. These vectors and
their coefficients define the strip parts and the colours of all edge
profiles.

This argument necessarily incurs an \(o(n)\) loss, so it cannot by itself
recover the exact AFL formula. The second layer of the proof addresses the
critical scale directly. We return to the dual vectors before repair and
compression and repeat the repair procedure with an error tending to zero.
This produces a finite exact covering configuration while changing the
average by only $o(1)$. Applying the same convex-hull compression as above,
the strip profile bound forces equality at the critical value. The equality
case is rigid: after permuting the colours, it forces positive integers
$c_1+\cdots+c_r=t+r-1$ and the axis vectors
$\frac{t}{c_1}e_1,\ldots,\frac{t}{c_r}e_r$.
Regularity converts these vectors into canonical edge profiles. A deficit
identity assigns the exceptional vertices to suitable colours, and a cyclic
switching argument finds an exact matching despite the remaining sparse
family of missing or wrongly coloured profile edges. This rigidity,
absorption, and switching mechanism is the essential new ingredient that
upgrades the structural \(o(n)\)-approximation to the exact bound.

\subsection{Hypergraph applications}

The structural theorem applies beyond complete hypergraphs. Our first
application is a sparse transference version of AFL. Many classical dense
results admit analogues in sparse random settings; examples include the work
of R\"odl and Ruci\'nski~\cite{RR} on Ramsey's theorem and the transference
of Tur\'an's theorem due to Conlon and Gowers~\cite{CG}, Balogh, Morris and
Samotij~\cite{BMS}, Saxton and Thomason~\cite{ST}, and
Schacht~\cite{Schacht}. Gishboliner, Glock, Michaeli and
Sgueglia~\cite{GGMS} proved the corresponding random-hypergraph AFL theorem
using the dense AFL theorem as a black box. Our structural reduction gives an
independent topology-free route, which together with the elementary matching bound for strip colourings yields the transference AFL theorem in $G^{(t)}(n,p)$.

\begin{theorem}\label{thm:randomaaa0}
Let $t,r\ge2$ and $\varepsilon>0$, and let
$G\sim G^{(t)}(n,p)$. There exists $C>0$ such that, if
$p\ge Cn^{-t+1}$, then with high probability the following holds. For every
$r$-edge-colouring $\chi$ of $G$, there exists an $r$-strip
$r$-edge-colouring $\psi$ of $K_n^{(t)}$ such that
\[
\nu(\chi)\ge\nu(\psi)-\varepsilon n.
\]
\end{theorem}

We also apply the robust critical theorem to stable Kneser hypergraphs. The
motivation goes back to Schrijver's theorem~\cite{Schrijver78} in 1978: the subgraph
of the Kneser graph induced by the $2$-stable $t$-sets has the same chromatic
number as the full Kneser graph and is vertex-critical. More generally,
regard $[n]$ as cyclically ordered, and call a $t$-set $s$-stable if every
two of its elements have cyclic distance at least $s$. Let
$\mathrm{KG}^q(n,t)_{s\text{-stab}}$ denote the subhypergraph of
$\mathrm{KG}^q(n,t)$ induced by the $s$-stable $t$-sets.

Ziegler~\cite{Ziegler2002} conjectured that the $q$-stable $q$-uniform
Kneser hypergraph has the same chromatic number as the full Kneser
hypergraph, and Alon, Drewnowski and {\L}uczak~\cite{ADL} made this
conjecture explicit. Meunier~\cite{Meunier} proposed the broader formula
\[
\chi\bigl(\mathrm{KG}^q(n,t)_{s\text{-stab}}\bigr)
=
\left\lceil
\frac{n-\max\{q,s\}(t-1)}{q-1}
\right\rceil
\qquad
\text{whenever }n\ge\max\{q,s\}t.
\]
Partial results were obtained by Meunier~\cite{Meunier},
Frick~\cite{Frick}, and others~\cite{Jonsson,Chen,Jafari,Daneshpajouh}.
We prove the conjectured value throughout the linear matching regime, even
when $s$ is a sufficiently small linear fraction of $q$. In this range the
$t$-graph of $s$-stable sets has minimum degree sufficiently close to that
of the complete $t$-graph, and hence falls within
Theorem~\ref{thm:robust-critical-AFL}.

\begin{theorem}\label{thm:stable}
Fix $t\ge2$ and $\gamma>0$. There exists $c=c(t,\gamma)>0$ such that the
following holds. Let the integer-valued functions $q=q(n)\ge2$ and
$s=s(n)\ge1$ satisfy
$\gamma n\le q, tq\le n,$ and $s\le cq.$ Then, for all sufficiently large $n$,
\[
\chi\bigl(\mathrm{KG}^q(n,t)_{s\text{-stab}}\bigr)
=
\left\lceil\frac{n-q(t-1)}{q-1}\right\rceil.
\]
\end{theorem}

Indeed, the $t$-graph of $s$-stable sets has minimum degree at least
$
\binom{n-1}{t-1}-O_t(sn^{t-2}).
$
Since $s\le cq\le cn/t$, this is sufficiently close to
$\binom{n-1}{t-1}$ when $c=c(t,\gamma)$ is sufficiently small, so the
lower bound follows from the robust critical theorem. The matching upper bound is given by the standard AFL construction.

\subsection{Graph tilings}

The second part of the paper shows that the LP-duality philosophy is not
limited to matchings. For a host graph $G$ and a fixed graph $H$, define
\[
Rt_r(H;G)
:=
\min_{\chi:E(G)\to[r]}
\max_{\ell\in[r]}
\{\text{maximum number of vertex-disjoint colour-$\ell$ copies of $H$}\},
\]
and write $\nu_H(\chi)$ for the corresponding maximum monochromatic
$H$-tiling size. For an integer $k\ge1$, let $R_r(kH)$ be the least $n$
such that every $r$-edge-colouring of $K_n$ contains a monochromatic copy of
$kH$.

The study of Ramsey numbers for graph tilings was initiated by Burr,
Erd\H{o}s and Spencer~\cite{BES} in 1975. They proved that, for every fixed
graph $H$ without isolated vertices,
\[
Rt_2(H;K_n)
=
\frac{n}{2v(H)-\alpha(H)}+O_H(1),
\]
or equivalently,
$R_2(kH)=\bigl(2v(H)-\alpha(H)\bigr)k+O_H(1)$.
Thus the general two-colour problem is governed by $\alpha(H)$. Subsequent
work refined the additive term and the threshold for exactness~\cite{Burr,BS},
while special multicolour cases were studied by Cockayne and
Lorimer~\cite{CL} and Loo~\cite{Loo}. Related questions for non-complete
hosts have been investigated under minimum-degree, random, and robust
pseudorandomness assumptions~\cite{BFT,ACFMPY,FMT,GKM,KM,NW}. A general
result covering every fixed number of colours and every fixed target graph
was not previously available.

We study this general multicolour problem. A first indication that the
two-colour formula does not extend directly is a sharp distinction between
bipartite and non-bipartite target graphs. For $r\ge3$ and connected non-bipartite $H$,
the elementary greedy lower bound is asymptotically optimal; see
Proposition~\ref{prop:nonbipartite}. For bipartite $H$, larger monochromatic
tilings can be forced, and the answer is governed by a finite optimization
over strip templates.

The graph analogue of Theorem~\ref{thm:informal-hypergraph-structure} is the
following structural theorem. For a graph $H$ with at least one edge, define
\[
m_2^*(H)
:=
\max\left\{
1,\,
\max\left\{
\frac{e(J)-1}{v(J)-2}:
J\subseteq H,\ v(J)\ge3
\right\}
\right\},
\]
where the inner maximum is ignored if no such subgraph $J$ exists.

\begin{theorem}\label{thm:graphstructure}
Let $r\ge2$, $\varepsilon>0$, and let $H$ be a fixed graph with at least one
edge. Then there exists $C>0$ such that, for all sufficiently large $n$, the
following holds. If $G\sim G(n,p)$ with
$p\ge Cn^{-1/m_2^*(H)}$, then with high probability every
$r$-edge-colouring $\chi$ of $G$ admits an $r$-strip
$r$-edge-colouring $\psi$ of $K_n$ satisfying
\[
\nu_H(\chi)\ge \nu_H(\psi)-\varepsilon n.
\]
\end{theorem}

The graph proof uses the same rounding-and-frequency repair and the same
face-Carath\'eodory compression as the hypergraph proof. The selected vectors
$R_1,\ldots,R_q$ retain many representative clusters, while the
Carath\'eodory coefficients are used directly as the strip-part proportions.
The additional difficulty is global consistency: the colours assigned to
pairs of strip types must ensure that every monochromatic copy of $H$ in the
reconstructed template corresponds to a genuine reduced $H$-pattern. We
obtain this by a multipartite Ramsey extraction inside and between the
representative classes.

For complete hosts, the structural theorem reduces the asymptotic
Ramsey-tiling problem to finitely many strip templates. Their pattern-packing
linear programs give the following effective solution.

\begin{theorem}\label{thm:finite-optimization}
Let $r\ge2$ and let $H$ be a fixed graph with at least one edge. Then there
exists a constant $\beta_{r,H}>0$, effectively computable from a finite
collection of rational linear programs whose sizes depend only on $r$ and
$H$, such that
\[
Rt_r(H;K_n)=(\beta_{r,H}+o(1))n.
\]
Consequently,
\[
R_r(kH)=\left(\frac1{\beta_{r,H}}+o(1)\right)k
\qquad\text{as }k\to\infty.
\]
\end{theorem}

The finite optimization does not always yield a transparent closed formula,
but we solve it explicitly in several natural families. We determine the
asymptotic value for connected non-bipartite graphs
(Proposition~\ref{prop:nonbipartite}), three-colour Hall-type bipartite graphs
(Theorem~\ref{thm:biparr=3}), complete bipartite graphs $K_{a,b}$ with four
and five colours (Theorems~\ref{thm:Kab-four} and~\ref{thm:Kab-five}),
balanced Hall-type bipartite graphs for every number of colours
(Theorem~\ref{thm:Kaa}), and the non-Hall double star $D_{2,3}$
(Theorem~\ref{thm:D23}). The last example shows that beyond the Hall-type
setting the answer can depend on finer features of $H$ than its bipartition
sizes.

\subsection{Summary of the results}

The following table summarizes the main structural reductions and their
applications.

\begin{table}[htbp]
  \centering
  \footnotesize
  \renewcommand{\arraystretch}{1.5}
  \begin{tabular}{
    >{\raggedright\arraybackslash}p{0.21\textwidth}
    >{\raggedright\arraybackslash}p{0.19\textwidth}
    >{\raggedright\arraybackslash}p{0.27\textwidth}
    >{\raggedright\arraybackslash}p{0.25\textwidth}
  }
    \toprule
    \textbf{Host setting} &
    \textbf{Target structure} &
    \textbf{Structural input} &
    \textbf{Main consequences} \\
    \midrule

    Complete or locally almost-complete $t$-graphs &
    Monochromatic matchings &
    Pre-compression dual weights, critical rigidity, canonical profiles,
    and exact switching &
    Theorem~\ref{thm:robust-critical-AFL} and Theorem~\ref{thm:stable} \\

    \midrule

    Pseudorandom and random $t$-graphs &
    Monochromatic matchings &
    Theorem~\ref{thm:hypergraphstructure}: LP duality, repair, and direct
    profile compression &
    Theorem~\ref{thm:randomaaa0} and the transference AFL theorem \\

    \midrule

    Random and complete graphs &
    Monochromatic $H$-tilings &
    Theorem~\ref{thm:graphstructure}: reduced $H$-patterns, LP duality,
    direct compression, and multipartite Ramsey extraction &
    Theorem~\ref{thm:finite-optimization}: finite LP computation of
    $\beta_{r,H}$ \\

    \midrule

    Complete graphs, explicit families &
    Non-bipartite and bipartite tilings &
    Strip-template analysis &
    Proposition~\ref{prop:nonbipartite} and
    Theorems~\ref{thm:biparr=3},~\ref{thm:D23},
    ~\ref{thm:Kab-four},~\ref{thm:Kab-five}, and~\ref{thm:Kaa} \\

    \bottomrule
  \end{tabular}
  \vspace{0.2cm}
  \caption{\footnotesize Main structural reductions and applications.}
  \label{tab:main_results}
\end{table}

\noindent\textbf{Additional note.}
Freschi, Martin and Treglown~\cite{FMT} independently obtained related
results on Ramsey $H$-tilings in random and complete graphs. Their random
result covers two and three colours, and for $r\ge4$ the range
$\chi(H)\ge r$; in the complete-host setting they obtain additive
$O(1)$ bounds for all $H$ when $r=3$ and for $\chi(H)\ge r$ when $r\ge4$.
Our structural theorem applies to every fixed $r$ and every fixed graph $H$
with at least one edge, but in general gives an $o(n)$ error term. We believe
that the absorber-and-switching mechanism used for our exact hypergraph
result can be adapted to graph tilings and may lead to exact results in this
setting as well.

\subsection{Organization}

Section~\ref{sec:proofoverview} gives a high-level overview of the structural
arguments and of the exact hypergraph upgrade. The first half concerns
monochromatic matchings in uniform hypergraphs. In Section~\ref{sec:main}, we
prove the strip-colouring reduction, Theorem~\ref{thm:hypergraphstructure}.
In Section~\ref{sec:hypapp}, we prove the robust critical theorem, deduce the
stable Kneser results for the stated regimes, and then derive
the sparse random transference theorem.

The second half concerns Ramsey graph tilings. In
Section~\ref{sec:graphstructure}, we prove the graph-tiling structural theorem,
Theorem~\ref{thm:graphstructure}. In Section~\ref{sec:graphapp}, we derive
the finite LP optimization and determine explicit values for several graph
families. Section~\ref{sec:conclusion} contains concluding remarks and further
directions. The appendices contain the five-colour $K_{a,b}$ calculation, the
double-star calculation, and the additional hypergraph-tiling construction.

\section{Proof overview}\label{sec:proofoverview}

We explain the asymptotic strip reduction (Section~\ref{sec:main}), the exact hypergraph upgrade (Section~\ref{sec:hypapp}), and the additional Ramsey extraction needed for
graph tilings (Section~\ref{sec:graphstructure}). The main structural pipeline is summarized in
Figure~\ref{fig:structural-pipeline}.

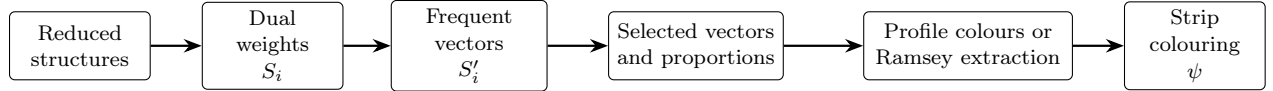
\begin{figure}[htbp]
\centering
\begin{tikzpicture}[
 >=Stealth,
 stage/.style={
  draw,
  rounded corners=2pt,
  align=center,
  minimum height=9mm,
  inner xsep=3pt,
  inner ysep=3pt,
  font=\scriptsize
 },
 flow/.style={->, line width=0.85pt}
]
\node[stage, text width=1.65cm] (reduced) at (0,0)
 {Reduced\\structures};
\node[stage, text width=1.65cm] (dual) at (2.55,0)
 {Dual weights\\$S_i$};
\node[stage, text width=1.85cm] (repair) at (5.15,0)
 {Frequent vectors\\$S'_i$};
\node[stage, text width=2.1cm] (profile) at (8.15,0)
 {Selected vectors\\and proportions};
\node[stage, text width=2.55cm] (template) at (11.75,0)
 {Profile colours or\\Ramsey extraction};
\node[stage, text width=1.65cm] (strip) at (14.75,0)
 {Strip colouring\\$\psi$};

\draw[flow] (reduced) -- (dual);
\draw[flow] (dual) -- (repair);
\draw[flow] (repair) -- (profile);
\draw[flow] (profile) -- (template);
\draw[flow] (template) -- (strip);
\end{tikzpicture}
\caption{\footnotesize The asymptotic structural pipeline. Regularity and LP duality
produce the vectors $S_i$; repair and Carath\'eodory compression select
at most $r$ frequent vectors together with the strip-part proportions. In the
matching setting the profiles are coloured directly, whereas graph tilings
require a multipartite Ramsey extraction to obtain a consistent pair
template. The robust exact argument returns to the pre-compression vectors
$S_i$.}
\label{fig:structural-pipeline}
\end{figure}

\subsection{Dual weights from reduced matchings}

Apply the multicolour weak hypergraph regularity lemma to an
$r$-edge-coloured pseudorandom $t$-graph. This gives an equipartition into
$M=O_{r,t,\varepsilon}(1)$ clusters. For each colour $\ell$, let
$\mathcal X_\ell$ be the family of regular cluster $t$-sets whose
colour-$\ell$ density is large. The fractional matching problem in
$\mathcal X_\ell$ has the dual fractional-cover program
\[
\min\sum_{i=1}^M z_i
\quad\text{subject to}\quad
\sum_{i\in X}z_i\ge1\quad(X\in\mathcal X_\ell),
\qquad z_i\ge0.
\]
A feasible reduced fractional matching lifts to an actual monochromatic
matching with only $o(n)$ loss, so the dual optimum controls the matching
size.

Scale an optimal dual solution in every colour by $t$, and group the weights
at cluster $i$ into
$S_i=(s_i^{(1)},\ldots,s_i^{(r)})\in[0,t]^r$. For all but $o(M^t)$ cluster
$t$-sets $X$, some coordinate satisfies
$\sum_{i\in X}s_i^{(\ell)}\ge t$. At the same time,
$\|\sum_iS_i\|_\infty$ records the largest dual objective. Thus the
colouring problem becomes a finite-dimensional covering problem.

\subsection{Repair and direct profile compression}

We round the vectors upwards to a fine grid and replace rare rounded values
by frequent ones. The resulting vectors $S'_1,\ldots,S'_M$ cover every
$t$-set of indices, every value has many representatives, and
$\|\sum_iS'_i\|_\infty\le\|\sum_iS_i\|_\infty+o(M)$.

Let $P$ be their convex hull and let $\overline S$ be their average. Move
from $\overline S$ in the negative all-ones direction until reaching the
boundary of $P$. The boundary point lies in a face of dimension at most
$r-1$, so Carath\'eodory's theorem gives frequent vectors
$R_1,\ldots,R_q$, where $q\le r$, and positive coefficients
$\lambda_1,\ldots,\lambda_q$ summing to one such that the boundary point is
$\sum_j\lambda_jR_j$. We partition the vertex set of $K_n^{(t)}$ into
parts of sizes $\lambda_jn+O(1)$. For each intersection profile, distinct
representatives of the required $R_j$ give a covering coordinate, which is
used as the colour of that profile.

This is a strip colouring. If $\mathcal M$ is a matching of colour $\ell$,
every edge of $\mathcal M$ consumes at least $t$ units of $\ell$-weight.
Hence its size is at most $n/t$ times the $\ell$-coordinate of the boundary
point, up to $O(1)$. Since the boundary point is coordinate-wise at most
$\overline S$, comparison with the lifted dual bound gives
$\nu(\chi)\ge\nu(\psi)-o(n)$.

\subsection{Critical rigidity and the robust exact theorem}

The exact argument follows the same general philosophy as the asymptotic
structural reduction, but the critical scale introduces a strong rigidity
phenomenon. Suppose that $n=(t+r-1)m+O(1)$ and every colour has matching
number at most $m+O(1)$. The corresponding dual vectors have coordinate
averages at most $t/(t+r-1)+o(1)$ and satisfy the covering condition on
almost every $t$-set.

We first apply a vanishing-error version of the repair procedure, obtaining a finite exact covering configuration
without changing the average by more than $o(1)$. We then apply the same
convex-hull compression as before. At the critical value, the strip bound
forces equality throughout the compression argument. This equality is rigid:
after permuting the coordinates, there are positive integers
$c_1,\ldots,c_r$ with $c_1+\cdots+c_r=t+r-1$ such that almost all dual
vectors are close to
\[
\frac{t}{c_1}e_1,\ldots,\frac{t}{c_r}e_r,
\]
where $e_1,\ldots,e_r$ is the standard basis of $\mathbb R^r$, with corresponding proportions $c_i/(t+r-1)$. Thus almost all clusters
split into $r$ canonical classes.

For colour $i$, the associated canonical profile uses $c_i$ vertices from
class $i$ and $c_j-1$ vertices from every other class. The dual constraints
and regularity imply that all but $o(m^t)$ edges of this profile are present
and have colour $i$. We regard the remaining edges as a sparse bad family.

We then absorb the exceptional vertices. After removing vertices of large
bad degree, the resulting class sizes satisfy an exact defect identity at the
critical order. This identity guarantees a colour in which enough exceptional
vertices can be covered greedily while leaving precisely the capacity needed
for the remaining canonical profile.

Finally, a switching argument eliminates the sparse bad family and produces
the required exact monochromatic matching. This first proves the result when
the host complement has maximum degree $o(n^{t-1})$; a diagonal argument
upgrades it to the fixed positive robustness margin in
Theorem~\ref{thm:robust-critical-AFL}. The same robustness also yields the
stable Kneser theorem throughout the corresponding locally robust range.

\subsection{The additional Ramsey step for graph tilings}

For graph tilings, the reduced linear programs are indexed by maps
$\phi:V(H)\to[M]$ whose edges are sent to dense regular pairs of one colour.
Their duals again give cluster-weight vectors. We apply the same repair and
face-Carath\'eodory argument, obtaining frequent vectors
$R_1,\ldots,R_q$ and positive coefficients
$\alpha_1,\ldots,\alpha_q$. The coefficients are used directly as the
strip-part proportions.

The remaining issue is global consistency. The colours assigned to pairs of
compressed types must ensure that every monochromatic map
$f:V(H)\to[q]$ is witnessed by a genuine reduced $H$-pattern. Each $R_j$
has many representative clusters. From every representative class we choose
a bounded set containing no exceptional pair, and a multipartite Ramsey
extraction makes all pairs inside and between the selected classes
monochromatic. These colours define the strip template $\lambda(i,j)$.
Choosing distinct representatives then turns every monochromatic map $f$
into a reduced $H$-pattern, whose dual constraint gives
$\sum_{x\in V(H)}R_{f(x)}^{(\ell)}\ge1$. Summing this inequality over a
tiling bounds its size by the corresponding coordinate of
$n\sum_j\alpha_jR_j$.

For complete hosts, a fixed strip template and a fixed vector of part
proportions give, in each colour, a finite fractional pattern-packing linear
program. Minimizing the largest of these optima over the finitely many
templates with at most $r$ parts defines $\beta_{r,H}$. Integral blow-ups of
a minimizing template give the upper bound, while the structural theorem
gives the matching lower bound up to $o(n)$; this proves
Theorem~\ref{thm:finite-optimization}.

\section{Structural reduction for hypergraph matchings}\label{sec:main}
A $t$-graph $\cG$ is said to be \emph{$(\eta,p,d,D)$-uniform} if, for every collection of $t$ pairwise disjoint vertex sets $X_1,\dots,X_t$ with $|X_i|\ge \eta n$ for all $i$, the number of edges having one vertex in each $X_i$ lies between
$dp\cdot\prod_{i=1}^t |X_i|$
 and 
$Dp\cdot\prod_{i=1}^t |X_i|.$
Let $\nu(\chi)$ be the maximum monochromatic matching size in $\chi$.

\begin{definition}\label{def:strip}
Let $t\geq2$ and $s\geq1$, and let $\cF$ be a $t$-graph. An edge-colouring $\chi$ of $\cF$ is an \emph{$s$-strip colouring} if there exists a partition
$V(\cF)=V_1\cup\cdots\cup V_s$ such that
\[
\chi(e_1)=\chi(e_2)
\quad\text{whenever}\quad
|e_1\cap V_j|=|e_2\cap V_j|
\ \text{for every }j\in[s].
\]
\end{definition}

\begin{theorem}\label{thm:hypergraphstructure}
Let $r,t\geq2$, $\varepsilon>0$ and $0<d\leq1<D$. There exists a constant $\eta>0$ such that the following holds. Let $n$ be sufficiently large and $n^{-t}\ll p\leq 1$, and let $\cG$ be an $(\eta,p,d,D)$-uniform $t$-graph on $n$ vertices equipped with an $r$-edge-colouring $\chi$. Then there exists an $r$-strip $r$-edge-colouring $\psi$ of $K_n^{(t)}$ such that 
$$\nu(\chi)\geq \nu(\psi) - \varepsilon n.$$
\end{theorem}

The proof is divided into the two steps outlined in Section~\ref{sec:proofoverview}.

\subsection{Preliminaries: Regularity and Embeddings}\label{sec:pre}
We begin with some standard notation. 
Let $\mathcal F=\mathcal F(V_1,\dots,V_t)$ be a $t$-partite $t$-graph. 
Denote by $e(V_1,\dots,V_t)$ the number of edges having exactly one vertex in each $V_i$. 
The \emph{density} of $\mathcal F$ is defined as
\[
d(V_1,\dots,V_t)
:= \frac{e(V_1,\dots,V_t)}{|V_1|\cdots |V_t|}.
\]

\begin{definition}
Let $\varepsilon > 0$ and $p\in[0,1]$. 
We say that a $t$-partite $t$-graph $\mathcal F(V_1,\dots,V_t)$ is \emph{$(\varepsilon,p)$-regular} 
if for every choice of subsets $V_i' \subseteq V_i$ with $|V_i'| \ge \varepsilon |V_i|$ for all $i\in[t]$, we have
\[
\bigl| d(V_1',\dots,V_t') - d(V_1,\dots,V_t) \bigr|
\le \varepsilon p.
\]
\end{definition}

We first present an embedding lemma for hyperedges in regular partitions.

\begin{lemma}\label{lem:distribution1_hygr}
Fix integers $2\le t\le M$ and 
$d>10\varepsilon> 0$. Let $n$ be sufficiently large, $n^{-t}\ll p\leq 1$ and let $k=\lfloor n/M\rfloor$. 
Let $\cF$ be a $t$-graph on $n$ vertices with an equipartition
\[
V(\cF)=V_1\cup\cdots\cup V_M,
\]
where $\bigl||V_i|-|V_j|\bigr|\le1$ for all $i,j\in[M]$. 
Let $\cX$ denote any collection of index sets $X=\{i_1,\ldots,i_t\}\subset [M]$ for which $(V_{i_1},\ldots,V_{i_t})$ is $(\varepsilon,p)$-regular and has density at least $d p$. 
If there exists a collection of non-negative values
$\{w_X\}_{X\in\mathcal X}$
such that for every $i\in[M]$
\[
\sum_{X\in \mathcal X: i\in X} w_X
\le
1-2\varepsilon,
\]
then $\mathcal F$ contains a matching of size
\[
k\cdot\sum_{X\in\mathcal X}w_X-O(1).
\]
\end{lemma}

\begin{proof}
For every \(X\in\mathcal X\), set
\(m_X:=\lfloor kw_X\rfloor\), and process the members of
\(\mathcal X\) in an arbitrary order. For \(i\in[M]\), let
$D_i:=\sum_{\substack{X\in\mathcal X\\i\in X}}m_X.$
The capacity assumption gives
\[
D_i
\le
k\sum_{\substack{X\in\mathcal X\\i\in X}}w_X
\le
(1-2\varepsilon)k.
\]

Suppose that \(X=\{i_1,\ldots,i_t\}\) is the next tuple to be processed.
For \(i\in X\), let \(U_i\subseteq V_i\) be the set of currently unused
vertices, and write
$u_i:=|V_i\setminus U_i|.$
Choose \(i_0\in X\) such that
\(s:=|U_{i_0}|=\min_{i\in X}|U_i|\). Since every previously embedded edge
whose reduced tuple contains \(i_0\) uses exactly one vertex from
\(V_{i_0}\), the total remaining quota involving \(i_0\) is
\(D_{i_0}-u_{i_0}\). In particular,
\[
m_X
\le
D_{i_0}-u_{i_0}
\le
(1-2\varepsilon)k-u_{i_0}
\le
s-2\varepsilon k,
\]
where the last inequality uses \(|V_{i_0}|\ge k\).

For every \(i\in X\), choose a subset \(U_i'\subseteq U_i\) of size \(s\).
We greedily select disjoint edges from the \(t\)-partite hypergraph induced
by these sets. Suppose that a maximal such matching has size
\(m<m_X\). After removing its vertices, each class contains a set
\(W_i\subseteq U_i'\) of size
$|W_i|=s-m>2\varepsilon k.$
For sufficiently large \(n\), this is at least
\(\varepsilon|V_i|\). Since the original tuple indexed by \(X\) is
\((\varepsilon,p)\)-regular and has density at least \(dp\), the tuple
\((W_{i_1},\ldots,W_{i_t})\) has density at least
\((d-\varepsilon)p>0\). It therefore contains an edge, contradicting the
maximality of the matching. Hence we can embed all \(m_X\) prescribed
edges. Repeating this for every \(X\in\mathcal X\) produces one matching
in \(\mathcal F\).

Finally,
\[
\sum_{X\in\mathcal X}m_X
\ge
k\sum_{X\in\mathcal X}w_X-|\mathcal X|
\ge
k\sum_{X\in\mathcal X}w_X-\binom Mt.
\]
Since \(M\) and \(t\) are fixed, the last term is
\(O_{M,t}(1)\), as required.
\end{proof}

Next we use a multicolour (weak) sparse regularity lemma, which is essentially Lemma~3.1 in~\cite{GGMS}.

\begin{lemma}[Multicolour sparse regularity lemma]\label{hreg}
Let $0\le d\le D$ and
$\eta\ll 1/Z\ll 1/m\ll\varepsilon\ll 1/D,1/r,1/t$. Then the following holds for $n$ sufficiently large and $n^{-t}\ll p\leq 1$. 

Assume $\cF$ is an $(\eta,p,d,D)$-uniform $t$-graph equipped with an $r$-edge-colouring $\chi$. Then there exists an equipartition $V(\cF)=(V_1,\ldots, V_{M})$, where $m\le M\le Z$, such that for all but at most $\varepsilon M^{t}$
$t$-tuples $X\in\binom{[M]}{t}$, the induced subhypergraph $\cF_\ell(V_X)$ is $(\varepsilon,p)$-regular for every colour $\ell\in[r]$.
\end{lemma}

We remark that the embedding Lemma~\ref{lem:distribution1_hygr} will be applied to the $\ell$-coloured $t$-graph $\mathcal G_\ell$, for each colour $\ell\in[r]$. Indeed, we will choose \[
\mathcal X_\ell
:=
\left\{X\in\binom{[M]}{t}:\ \mathcal G_\ell[V_X] \text{ is $(\varepsilon, p)$-regular with density at least $\frac{dp}{r}$}\right\}.
\]
For every non-exceptional $X\in\binom{[M]}{t}$, the clusters $V_i$, $i\in X$, have size at least $\eta n$, and hence the uniformity of $\cG$ gives
$\sum_{\ell=1}^r d_{\cG_\ell}(V_X)=d_{\cG}(V_X)\ge dp.$
Thus $d_{\cG_\ell}(V_X)\ge dp/r$ for some $\ell\in[r]$, and therefore
\[
\left|\binom{[M]}{t}\setminus\bigcup_{\ell=1}^r\mathcal X_\ell\right|
\le \varepsilon M^t.
\]

\subsection{Step I: Dual weights from reduced matchings}

The following lemma bounds the monochromatic matching numbers by the optima
of $r$ reduced linear programs.

\begin{lemma}\label{lem:chitoskeleton_hygr}
Let
$0<\eta\ll1/Z\ll1/m\ll\varepsilon\ll d/r,1/D,1/r,1/t$.
For $n$ sufficiently large and $n^{-t}\ll p\le1$, let $\mathcal G$ be an
$(\eta,p,d,D)$-uniform $t$-graph with an $r$-edge-colouring $\chi$, and let
$V(\cG)=V_1\cup\cdots\cup V_M$ be the equipartition given by
Lemma~\ref{hreg}. Put $k=\lfloor n/M\rfloor$. Then, for every
$\ell\in[r]$,
\[
\nu(\mathcal G_\ell)\ge kZ_\ell-3\varepsilon n,
\]
where $Z_\ell$ is the optimal value of the linear program 
\begin{equation}\label{eq:DL_hygr}
\begin{aligned}
\text{minimize}\quad 
& \sum_{i=1}^M z_i\\
\text{subject to}\quad
& \sum_{i:i\in X} z_i \ge 1
\qquad\text{for every } X\in\mathcal X_\ell,\\
& z_i \ge 0
\qquad\text{for every } i\in[M].
\end{aligned}
\tag{$D_\ell$}
\end{equation}
\end{lemma}

We may choose an optimal solution with $z_i\le1$ for every $i\in[M]$, since replacing any coordinate larger than $1$ by $1$ preserves all constraints.

\begin{proof}
Consider the dual linear program for fractional matchings:
\begin{equation}\label{eq:PL_hygr}
\begin{aligned}
\text{maximize}\quad 
& \sum_{X\in\mathcal X_\ell} w_X\\
\text{subject to}\quad
& \sum_{\substack{X\in\mathcal X_\ell: i\in X}} w_X \le 1
\qquad\text{for every } i\in[M],\\
& w_X \ge 0
\qquad\text{for every } X\in\mathcal X_\ell .
\end{aligned}
\tag{$P_\ell$}
\end{equation}
Let $\{w_X\}_{X\in\mathcal X_\ell}$ be an 
optimal solution with optimal value $W_\ell$. By strong LP duality, $W_\ell=Z_\ell$.
Apply Lemma~\ref{lem:distribution1_hygr}, with density parameter $d/r$, to the weights $\{(1-2\varepsilon)w_X\}_{X\in\mathcal X_\ell}$. We obtain
\[
\nu(\cG_\ell)
\ge
(1-2\varepsilon)kW_\ell-O(1)
\ge
kW_\ell-3\varepsilon n
=
kZ_\ell-3\varepsilon n,
\]
where the last inequality holds for sufficiently large $n$.
\end{proof}

\subsection{Step II: Repair and profile compression}

For each colour $\ell\in[r]$, let
$\{z_i^{(\ell)}\}_{i\in[M]}$ be an optimal solution given by
Lemma~\ref{lem:chitoskeleton_hygr}, and set
$S_i=(tz_i^{(1)},\ldots,tz_i^{(r)})$.

\begin{observation}\label{lem:initialvalues}
The vectors $S_1,\ldots,S_M\in[0,t]^r$ satisfy the following properties.
\begin{enumerate}
\item[(i)] For all but at most $\varepsilon M^t$ sets
$X\in\binom{[M]}t$, there is a colour $\ell\in[r]$ such that
$$\sum_{i\in X}s_i^{(\ell)}\ge t.$$
\item[(ii)] For every $\ell\in[r]$,
$$\sum_{i=1}^Ms_i^{(\ell)}=tZ_\ell.$$
In particular, writing $Z:=(Z_1,\ldots,Z_r)$, we have
$
\left\|\sum_i S_i\right\|_\infty=t\|Z\|_\infty.
$
\end{enumerate}
\end{observation}

The next lemma repairs the exceptional sets and keeps only frequent rounded
values.

\begin{lemma}\label{lem:op1}
Let $1/M \ll \varepsilon\ll1/r,1/t$. Let
$S_1,\ldots,S_M\in[0,t]^r$ satisfy
Observation~\ref{lem:initialvalues}\textnormal{(i)}. There are vectors
$S'_1,\ldots,S'_M\in[0,2t]^r$ such that:
\begin{enumerate}
\item For each $X\in\binom{[M]}{t}$, there exists a colour $\ell\in[r]$ such that 
$\sum_{i\in X} s_i'^{(\ell)}\geq t$.
\item For every $i \in [M]$, there are at least $t$ indices $j \in [M]$ such that $S'_j = S'_i$.
\item
\[
\bigg\|\sum_{i=1}^MS'_i\bigg\|_\infty
\le
\bigg\|\sum_{i=1}^MS_i\bigg\|_\infty
+3Mt\varepsilon^{1/(3tr)}.
\]
\end{enumerate}
\end{lemma}

\begin{proof}
Set $\delta:=t\varepsilon^{1/(3tr)}$ and $L:=t\varepsilon^{1/t}M$. We first replace every coordinate of $S_i$ by the smallest multiple of $\delta$ not smaller than it. For every $i\in[M]$, set the discretized vector $\tilde{S}_i:= \delta\cdot\left\lceil \frac{S_i}{\delta}\right\rceil$.

Note that the vectors $\{\tilde{S}_i\}$ take at most $(2t/\delta)^r$ distinct values. Indeed, the set 
$$C=\{0,\delta,2\delta,\ldots,\lceil t/\delta\rceil\delta\}\subset [0,2t]$$ 
 contains at most $\lceil t/\delta\rceil+1\leq (2 t/\delta)$ values, and $\tilde{S}_i\in C^r$. Moreover,
$\left\|\sum_i\widetilde S_i\right\|_\infty
\le
\left\|\sum_iS_i\right\|_\infty+M\delta.$

We next construct $\{S'_i\}$ by selecting the vectors that appear sufficiently many times among $\{\tilde{S}_i\}$. We keep these frequently occurring vectors and use their values to replace the vectors that appear only a few times. Set
$$\cS:=\bigl\{S\in\{\tilde{S}_i\}_{i\in[M]}:\ |\{i\in[M]: \tilde{S}_i=S\}|\ge L\bigr\}.$$
For sufficiently small $\varepsilon$, we have $(2t/\delta)^rL<M$, so $\cS$ is non-empty. We define $\{S'_i\}$ as follows: For each $i\in [M]$, if $\tilde{S}_i\in \cS$, then set $S'_i:=\tilde{S}_i$; otherwise choose an arbitrary $S\in \cS$ and set $S'_i:=S$. Then \textnormal{(2)} holds immediately for $\{S'_i\}$, since $L>t$, and it remains to verify \textnormal{(1)} and \textnormal{(3)}.

For (1), suppose that some $X\in\binom{[M]}{t}$ satisfies
\[
\sum_{i\in X}s_i'^{(\ell)}<t
\qquad\text{for every }\ell\in[r].
\]
Let $R_1,\ldots,R_h$ be the distinct values in the multiset $\{S_i':i\in X\}$, with multiplicities $a_1,\ldots,a_h$, where $a_1+\cdots+a_h=t$. For each $j\in[h]$, let
$I_j:=|\{u\in[M]:\widetilde S_u=R_j\}|.
$
Since every $R_j$ is frequent, $I_j\ge L$. Every $t$-set whose multiset of rounded values consists of $a_j$ copies of $R_j$ for each $j$ also violates the covering condition for the original vectors, because $S_u\le\widetilde S_u$ coordinatewise. The number of such $t$-sets is
\[
\prod_{j=1}^h\binom{I_j}{a_j}
\ge
\prod_{j=1}^h\binom{\lfloor L\rfloor}{a_j}
=
(1+o(1))\frac{L^t}{\prod_{j=1}^h a_j!}
>
\varepsilon M^t.
\]
Indeed, $L=t\varepsilon^{1/t}M$ and
$t^t/\prod_j a_j!\ge t^t/t!>1$. This contradicts
Observation~\ref{lem:initialvalues}\textnormal{(i)}.

For (3), the number of indices $i\in[M]$ for which $\widetilde S_i\notin\cS$ is at most
$\left(\frac{2t}{\delta}\right)^rL
\le
\varepsilon^{1/(2t)}M,$
provided that $\varepsilon$ is sufficiently small. Hence
\[
\max_{\ell\in[r]}
\sum_{i=1}^M
|s_i'^{(\ell)}-\widetilde s_i^{(\ell)}|
\le
2t\varepsilon^{1/(2t)}M.
\]
Together with the preceding estimate, this gives
\begin{align*}
\left\|\sum_{i=1}^M S'_i\right\|_\infty
\le
\left\|\sum_{i=1}^M\widetilde S_i\right\|_\infty
+
2t\varepsilon^{1/(2t)}M
&\le
\left\|\sum_{i=1}^M S_i\right\|_\infty
+M\delta+2t\varepsilon^{1/(2t)}M\\
&\le
\left\|\sum_{i=1}^M S_i\right\|_\infty
+3Mt\varepsilon^{1/(3tr)}\tag*{\qedhere}
\end{align*}
\end{proof}

We use the following form of Carath\'eodory's theorem; see, for instance,
\cite{ziegler}.

\begin{theorem}[Carath\'eodory's theorem]\label{thm:Cara}
Let \(K\subseteq\mathbb R^r\) be compact and let
\(x\in\partial\operatorname{conv}(K)\). Then \(x\) is a convex
combination of at most \(r\) points of \(K\).
\end{theorem}
\begin{lemma}\label{lem:direct-profile-compression}
Let $S'_1,\ldots,S'_M$ be the vectors obtained from
Lemma~\ref{lem:op1}. Then there exists an $r$-strip
$r$-edge-colouring $\psi$ of $K_n^{(t)}$ such that
\[
\nu(\psi)
\le
\frac{n}{Mt}\left\|\sum_{i=1}^M S'_i\right\|_\infty
+O_{r,t}(1).
\]
\end{lemma}

\begin{proof}
Let $P:=\operatorname{conv}\{S'_1,\ldots,S'_M\}$ and
$\overline S:=M^{-1}\sum_{i=1}^M S'_i$. Define
\[
\alpha:=\max\{a\ge0:\overline S-a\mathbbm1_r\in P\}
\qquad\text{and}\qquad
J:=\overline S-\alpha\mathbbm1_r.
\]
By the maximality of $\alpha$, we have $J\in\partial P$. Hence, by
Theorem~\ref{thm:Cara}, there are distinct values
$R_1,\ldots,R_q$ occurring among the $S'_i$, where $q\le r$, and positive
numbers $\lambda_1,\ldots,\lambda_q$ summing to one such that
$J=\sum_{j=1}^q\lambda_jR_j$.

Partition $V(K_n^{(t)})$ into sets $U_1,\ldots,U_q$ such that
$|U_j|=\lambda_jn+O(1)$ for every $j\in[q]$. Consider a profile
$\mathbf a=(a_1,\ldots,a_q)\in\mathbb Z_{\ge0}^q$ with
$\sum_{j=1}^q a_j=t$. Since every $R_j$ has at least $t$ representatives, we may choose
$t$ distinct indices consisting of $a_j$ indices carrying the value $R_j$
for each $j\in[q]$. Let $X$ be the set of these indices. Then
Lemma~\ref{lem:op1}\textnormal{(1)}, applied to $X$, gives a colour
$\ell(\mathbf a)\in[r]$ such that
\[
\sum_{j=1}^q a_jR_j^{(\ell(\mathbf a))}\ge t.
\] 
Colour every edge of profile $\mathbf a$ with colour $\ell(\mathbf a)$.
This defines an $r$-strip $r$-edge-colouring $\psi$ of $K_n^{(t)}$.

Let $\mathcal M$ be a matching of colour $\ell$. By the definition of
$\psi$, every $e\in\mathcal M$ satisfies
$\sum_{j=1}^q|e\cap U_j|R_j^{(\ell)}\ge t$. Since the edges of
$\mathcal M$ are pairwise disjoint,
\[
t|\mathcal M|
\le
\sum_{e\in\mathcal M}\sum_{j=1}^q|e\cap U_j|R_j^{(\ell)}
\le
\sum_{j=1}^q|U_j|R_j^{(\ell)}
=
nJ_\ell+O_{r,t}(1).
\]
Since $J\le\overline S$ coordinatewise, it follows that
\[
|\mathcal M|
\le
\frac{n}{t}\overline S_\ell+O_{r,t}(1)
\le
\frac{n}{Mt}\left\|\sum_{i=1}^M S'_i\right\|_\infty
+O_{r,t}(1).
\]
Taking the maximum over all colours proves the result.
\end{proof}

\subsection{Proof of Theorem~\ref{thm:hypergraphstructure}}

\begin{proof}[Proof of Theorem~\ref{thm:hypergraphstructure}]
Fix $\varepsilon>0$. Choose $\rho>0$ such that
$3\rho+3\rho^{1/(3tr)}<\varepsilon/2$, and then choose the remaining
parameters according to
\[
0<\eta\ll1/Z\ll1/m\ll\rho\ll d/r,1/D,1/r,1/t.
\]
We may assume that $\rho^{1/t}m\ge1$. Apply Lemma~\ref{hreg}, with $\rho$
in place of its regularity parameter, to obtain an equipartition into
$M$ clusters, where $m\le M\le Z$, and put $k=\lfloor n/M\rfloor$.

Lemma~\ref{lem:chitoskeleton_hygr}, again with $\rho$ in place of
$\varepsilon$, gives
$\nu(\chi)\ge k\|Z\|_\infty-3\rho n$. Let $S_1,\ldots,S_M$ be the vectors
from Observation~\ref{lem:initialvalues}, and apply Lemma~\ref{lem:op1}.
Every value among the resulting vectors $S'_1,\ldots,S'_M$ occurs at least
$t\rho^{1/t}M\ge t$ times. Lemma~\ref{lem:direct-profile-compression}
therefore gives an $r$-strip colouring $\psi$ with
\[
\begin{aligned}
\nu(\psi)
\le
\frac{n}{Mt}\left\|\sum_iS'_i\right\|_\infty+O(1)
\le
\frac nM\|Z\|_\infty+3\rho^{1/(3tr)}n+O(1)
\le
k\|Z\|_\infty+3\rho^{1/(3tr)}n+O(1).
\end{aligned}
\]
For the last inequality, use $0\le n/M-k<1$ and
$\|Z\|_\infty\le M$. Combining the two estimates and taking $n$
sufficiently large gives $\nu(\chi)\ge\nu(\psi)-\varepsilon n$.
\end{proof}

\section{A robust critical theorem and hypergraph applications}
\label{sec:hypapp}

For a $t$-graph $\cG$, write $\overline{\cG}$ for its complement in the
complete $t$-graph on $V(\cG)$. If $\chi$ is an $r$-edge-colouring, let
$\nu_i(\chi)$ denote the matching number in colour $i$. We first prove
Theorem~\ref{thm:robust-critical-AFL}, and then derive its 
stable Kneser consequences before turning to random transference.

\subsection{Proof of the robust critical theorem}
\label{subsec:robust-critical}
We use only the dual vectors
$S_i=(tz_i^{(1)},\ldots,tz_i^{(r)})$
from Observation~\ref{lem:initialvalues}, before the repair and compression
steps in Section~\ref{sec:main}. For the proof it is more convenient to work with the complement. Since
$d_{\overline{\cG}}(v)=\binom{n-1}{t-1}-d_{\cG}(v)$ and
$\binom{n-1}{t-1}=\Theta_t(n^{t-1})$, after changing the constant $c$
if necessary, Theorem~\ref{thm:robust-critical-AFL} is equivalent to the
following statement: there exist $c>0$ and $n_0$ such that every
$t$-graph $\cG$ on $n\ge n_0$ vertices satisfying
$\Delta(\overline{\cG})\le cn^{t-1}$ satisfies the conclusion of the
theorem. We use this equivalent formulation throughout the proof.

We first record the elementary strip bound
used both here and in the transference argument.

\begin{lemma}\label{lem:aaalower}
Let $q\le r$. Every $r$-edge-coloured $q$-strip colouring $\varphi$ of
$K_n^{(t)}$ satisfies
\[
\nu(\varphi)\ge\left\lfloor\frac{n}{t+r-1}\right\rfloor.
\]
\end{lemma}

\begin{proof}
Put $m=\lfloor n/(t+r-1)\rfloor$, and let
$U_1,\ldots,U_q$ be the strip parts. The result is trivial if $m=0$.
Set $c_j=\lfloor |U_j|/m\rfloor$. Then
\[
\sum_{j=1}^qc_j>\frac{n}{m}-q\ge t+r-1-q\ge t-1,
\]
so $\sum_jc_j\ge t$. Choose non-negative integers $a_j\le c_j$ with
$\sum_ja_j=t$. Since $|U_j|\ge a_jm$ for every $j$, the complete profile
hypergraph with profile $(a_1,\ldots,a_q)$ contains $m$ disjoint edges,
all of the same colour.
\end{proof}

Let $e_1,\ldots,e_r$ be the standard basis of $\mathbb R^r$.

\begin{lemma}\label{lem:critical-rigidity}
Let $M_h\to\infty$, and let
$Y_{h,1},\ldots,Y_{h,M_h}\in[0,t]^r$. Suppose that
\[
\bigg|\bigg\{Q\in\binom{[M_h]}t:
\max_{i\in[r]}\sum_{u\in Q}Y_{h,u}^{(i)}<t\bigg\}\bigg|
=o(M_h^t)
\]
and, uniformly in $i\in[r]$,
\[
\frac1{M_h}\sum_{u=1}^{M_h}Y_{h,u}^{(i)}
\le\frac{t}{t+r-1}+o(1).
\]
After passing to a subsequence and permuting the coordinates, there are
fixed positive integers $c_1,\ldots,c_r$ with
$\sum_ic_i=t+r-1$, numbers $\xi_h\downarrow0$, and pairwise disjoint sets
$\cC_{h,1},\ldots,\cC_{h,r}\subseteq[M_h]$ such that
\[
\left|[M_h]\setminus\bigcup_{i=1}^r\cC_{h,i}\right|=o(M_h),
\qquad
|\cC_{h,i}|=\frac{c_i}{t+r-1}M_h+o(M_h),
\]
and
\[
\left\|Y_{h,u}-\frac{t}{c_i}e_i\right\|_\infty\le\xi_h
\qquad (u\in\cC_{h,i}).
\]
\end{lemma}

\begin{proof}
Put $L=t+r-1$. We first repeat the repair procedure from
Lemma~\ref{lem:op1}, now with an error parameter tending to zero. Set
\[
\beta_h:=
\frac1{M_h^t}
\bigg|\bigg\{Q\in\binom{[M_h]}t:
\max_{i\in[r]}\sum_{u\in Q}Y_{h,u}^{(i)}<t\bigg\}\bigg|.
\]
Thus $\beta_h\to0$. After discarding finitely many terms, define
$\varepsilon_h:=\max\{\beta_h^{1/2},M_h^{-t/2}\}$. Then
$\varepsilon_h\to0$, $\beta_h\le\varepsilon_h$, and
$\varepsilon_h^{1/t}M_h\to\infty$.

Set $\delta_h:=t\varepsilon_h^{1/(3tr)}$ and
$F_h:=t\varepsilon_h^{1/t}M_h$. Round every coordinate of $Y_{h,u}$
upwards to the nearest multiple of $\delta_h$, and denote the resulting
vector by
\[
\widetilde Y_{h,u}
:=
\delta_h\left\lceil\frac{Y_{h,u}}{\delta_h}\right\rceil.
\]
The vectors $\widetilde Y_{h,u}$ lie in $[0,2t]^r$ and take at most
$N_h:=(2t/\delta_h)^r$ distinct values. Call a rounded value frequent if
it occurs at least $F_h$ times. Since
\[
N_hF_h
\le O_{r,t}\bigl(\varepsilon_h^{2/(3t)}M_h\bigr)
=o(M_h),
\]
at least one frequent value exists for all sufficiently large $h$.

Define $Y'_{h,u}$ as follows. If $\widetilde Y_{h,u}$ is frequent, set
$Y'_{h,u}:=\widetilde Y_{h,u}$; otherwise replace
$\widetilde Y_{h,u}$ by an arbitrary frequent rounded value. Let
$\cR_h\subseteq[M_h]$ be the set of replaced indices. Then
$|\cR_h|\le N_hF_h=o(M_h)$, and every value occurring among the
$Y'_{h,u}$ occurs at least $F_h$ times.

We claim that the repaired vectors cover every $t$-set:
\[
\max_{\ell\in[r]}\sum_{u\in Q}Y_{h,u}'^{(\ell)}\ge t
\qquad\text{for every }Q\in\binom{[M_h]}t.
\]
Suppose otherwise, and let $R_1,\ldots,R_s$ be the distinct values in the
multiset $\{Y'_{h,u}:u\in Q\}$, with multiplicities
$a_1,\ldots,a_s$, where $\sum_ja_j=t$. Let
$I_j:=|\{u\in[M_h]:\widetilde Y_{h,u}=R_j\}|$. Since every $R_j$ is
frequent, $I_j\ge F_h$. Every $t$-set whose multiset of rounded values
contains $a_j$ copies of $R_j$ for every $j$ violates the original
covering condition, because $Y_{h,u}\le\widetilde Y_{h,u}$
coordinatewise. Hence the number of violating $t$-sets is at least
\[
\prod_{j=1}^s\binom{I_j}{a_j}
\ge
\prod_{j=1}^s\binom{\lfloor F_h\rfloor}{a_j}
=
(1+o(1))\frac{F_h^t}{\prod_{j=1}^sa_j!}
>
\varepsilon_hM_h^t.
\]
Here we used $F_h^t=t^t\varepsilon_hM_h^t$ and
$t^t/\prod_ja_j!\ge t^t/t!>1$. This contradicts
$\beta_h\le\varepsilon_h$ and proves the claim.

The repair changes the original array by $o(1)$ on average. Indeed,
$Y'_{h,u}=\widetilde Y_{h,u}$ for $u\notin\cR_h$, and hence
\[
\frac1{M_h}\sum_{u=1}^{M_h}
\|Y'_{h,u}-Y_{h,u}\|_\infty
\le
\delta_h+\frac{2t|\cR_h|}{M_h}
=o(1).
\]
Let $\overline Y'_h:=M_h^{-1}\sum_uY'_{h,u}$. It follows that, uniformly
in $i\in[r]$,
\[
\overline Y_h'^{(i)}
\le\frac{t}{L}+o(1).
\]

We now apply the same finite profile compression as in
Lemma~\ref{lem:direct-profile-compression}. Let
\[
\cV_h:=\{Y'_{h,u}:u\in[M_h]\},
\qquad
P_h:=\operatorname{conv}(\cV_h),
\]
and define
\[
\alpha_h:=\max\{a\ge0:\overline Y'_h-a\mathbbm1_r\in P_h\},
\qquad
J_h:=\overline Y'_h-\alpha_h\mathbbm1_r.
\]
By maximality, $J_h\in\partial P_h$. Theorem~\ref{thm:Cara} gives
$q_h\le r$, distinct values $W_{h,1},\ldots,W_{h,q_h}\in\cV_h$, and
positive coefficients $\lambda_{h,1},\ldots,\lambda_{h,q_h}$ summing to
one such that
\[
J_h=\sum_{j=1}^{q_h}\lambda_{h,j}W_{h,j}.
\]

Consider an integer $t$-profile
$a=(a_1,\ldots,a_{q_h})$ with $\sum_ja_j=t$. Since every value in
$\cV_h$ occurs at least $F_h\ge t$ times, we may choose $t$ distinct
indices consisting of $a_j$ indices carrying $W_{h,j}$ for every $j$.
The exact covering property of the repaired array therefore gives a colour
$\ell_h(a)\in[r]$ such that
\[
\sum_{j=1}^{q_h}a_jW_{h,j}^{(\ell_h(a))}\ge t.
\]

For each $N$, partition $V(K_N^{(t)})$ into sets
$U_1,\ldots,U_{q_h}$ with
$|U_j|=\lambda_{h,j}N+O(1)$, and colour every edge of profile $a$ by
$\ell_h(a)$. This is a $q_h$-strip $r$-edge-colouring. If $\cM$ is a
matching of colour $i$, then
\[
t|\cM|
\le
\sum_{e\in\cM}\sum_{j=1}^{q_h}|e\cap U_j|W_{h,j}^{(i)}
\le
\sum_{j=1}^{q_h}|U_j|W_{h,j}^{(i)}
=
NJ_h^{(i)}+O_{r,t}(1).
\]
Since $q_h\le r$, the elementary strip bound above gives a monochromatic
matching of size $\lfloor N/L\rfloor$. Letting $N\to\infty$ gives
$\|J_h\|_\infty\ge t/L$. Since
$\|\overline Y'_h\|_\infty\le t/L+o(1)$ and
$\|J_h\|_\infty=\|\overline Y'_h\|_\infty-\alpha_h$, we obtain
\[
\alpha_h=o(1),
\qquad
\|J_h\|_\infty=\frac{t}{L}+o(1),
\qquad
\|\overline Y'_h-J_h\|_\infty=o(1).
\]

Pass to a subsequence on which $q_h=q_0$ is constant and, after relabelling,
\[
W_{h,j}\to W_j,
\qquad
\lambda_{h,j}\to\lambda_j
\qquad (j\in[q_0]).
\]
Discard the indices for which $\lambda_j=0$, and relabel the remaining
indices as $[q]$. Thus $\lambda_j>0$, $\sum_{j=1}^q\lambda_j=1$, and
\[
J:=\lim_{h\to\infty}J_h
=
\sum_{j=1}^q\lambda_jW_j.
\]
Moreover, $J_i\le t/L$ for every $i\in[r]$ and
$\|J\|_\infty=t/L$.

There are only finitely many integer $t$-profiles on $q$ parts. After one
further subsequence, the covering colour $\ell_h(a)$ is independent of $h$
for every such profile. Consequently, every integer $t$-profile
$a=(a_1,\ldots,a_q)$ is covered by some coordinate $\ell(a)\in[r]$:
\[
\sum_{j=1}^qa_jW_j^{(\ell(a))}\ge t.
\]

We now determine this finite limiting configuration. Set
$d_j:=L\lambda_j$. We claim that every integer $t$-profile $a$ satisfies
$a_j\ge d_j$ for some $j$. Otherwise
$\rho:=\min_{j:a_j>0}\lambda_j/a_j>1/L$, and hence
\[
J_{\ell(a)}
=
\sum_{j=1}^q\lambda_jW_j^{(\ell(a))}
\ge
\rho\sum_{j=1}^qa_jW_j^{(\ell(a))}
\ge
\rho t
>
\frac{t}{L},
\]
contradicting $J_{\ell(a)}\le t/L$.

Let $b_j:=\lceil d_j\rceil-1$. If $q<r$, then
$\sum_jb_j\ge L-q\ge t$, so there is an integer $t$-profile satisfying
$a_j\le b_j<d_j$ for every $j$, a contradiction. Thus $q=r$. If some
$d_j$ is non-integral, then
$\sum_j(\lceil d_j\rceil-d_j)$ is a positive integer, and therefore
$\sum_jb_j\ge t$. This again gives a $t$-profile with $a_j<d_j$ for every
$j$. Hence
\[
d_j=c_j\in\mathbb N,
\qquad
c_1+\cdots+c_r=L,
\qquad
\lambda_j=\frac{c_j}{L}.
\]

For $i\in[r]$, define the profile $p^{(i)}$ by
\begin{equation}\label{eq:canonical-profile}
p_i^{(i)}=c_i,
\qquad
p_j^{(i)}=c_j-1\quad(j\ne i).
\end{equation}
Let $\ell_i$ cover this profile. Since
$\sum_jc_jW_j^{(\ell_i)}=LJ_{\ell_i}\le t$, we have
\[
t
\le
\sum_{j=1}^rp_j^{(i)}W_j^{(\ell_i)}
=
\sum_{j=1}^rc_jW_j^{(\ell_i)}
-
\sum_{j\ne i}W_j^{(\ell_i)}
\le
t.
\]
It follows that $W_j^{(\ell_i)}=0$ for $j\ne i$ and
$W_i^{(\ell_i)}=t/c_i$. The colours
$\ell_1,\ldots,\ell_r$ are distinct. After permuting the coordinates,
\[
W_i=\frac{t}{c_i}e_i\quad(i\in[r]),
\qquad
J=\frac{t}{L}\mathbbm1_r.
\]

Since $q=r$ and $q_h\le r$, the preceding limiting representation shows
that, after relabelling, $q_h=r$ for all sufficiently large $h$, and
\[
W_{h,i}\to\frac{t}{c_i}e_i,
\qquad
\lambda_{h,i}\to\frac{c_i}{L}
\qquad(i\in[r]).
\]
Set
\[
\zeta_h:=
\max_{i\in[r]}
\left\|W_{h,i}-\frac{t}{c_i}e_i\right\|_\infty,
\]
so $\zeta_h\to0$.

We next show that almost every repaired vector is close to one of these
$r$ axis vectors. Fix $u\in[M_h]$. Since every value in $\cV_h$ occurs
at least $F_h\ge t$ times, we may choose distinct representatives
consisting of $Y'_{h,u}$ together with $c_i-1$ copies of $W_{h,i}$ for
every $i\in[r]$. This is a collection of
$1+\sum_i(c_i-1)=t$ repaired vectors. The exact covering property gives
some $k\in[r]$ such that
\[
Y_{h,u}'^{(k)}
+
\sum_{i=1}^r(c_i-1)W_{h,i}^{(k)}
\ge t.
\]
Since $W_{h,i}\to te_i/c_i$, uniformly in $k$,
\[
Y_{h,u}'^{(k)}
\ge
\frac{t}{c_k}-o(1).
\]

Put $g_{h,u}:=\sum_{i=1}^rc_iY_{h,u}'^{(i)}$. The preceding inequality
implies $g_{h,u}\ge t-o(1)$ uniformly in $u$. On the other hand,
$\overline Y'_h\to t\mathbbm1_r/L$, and hence
\[
\frac1{M_h}\sum_{u=1}^{M_h}g_{h,u}
=
\sum_{i=1}^rc_i\overline Y_h'^{(i)}
=
t+o(1).
\]
It follows that there is a sequence $\rho_h\downarrow0$ such that
$g_{h,u}\le t+\rho_h$ for all but $o(M_h)$ indices $u$.

For each such index, choose $k$ with
$Y_{h,u}'^{(k)}\ge t/c_k-o(1)$. Since
$g_{h,u}\le t+\rho_h$, we also have
$Y_{h,u}'^{(k)}\le t/c_k+o(1)$ and
$\sum_{i\ne k}c_iY_{h,u}'^{(i)}=o(1)$. Therefore
\[
\min_{i\in[r]}
\left\|Y'_{h,u}-\frac{t}{c_i}e_i\right\|_\infty=o(1)
\]
for all but $o(M_h)$ indices, with the error uniform over those indices.

Choose $\tau_h\downarrow0$ such that all but $o(M_h)$ repaired vectors lie
within distance $\tau_h$ of one of the points $te_i/c_i$. For large $h$
these neighbourhoods are pairwise disjoint. Let $\cD_{h,i}$ be the set of
indices whose repaired vector lies within distance $\tau_h$ of
$te_i/c_i$. Then
\[
\left|[M_h]\setminus\bigcup_{i=1}^r\cD_{h,i}\right|=o(M_h).
\]
Since $\overline Y'_h\to t\mathbbm1_r/L$, for every $i\in[r]$,
\[
\overline Y_h'^{(i)}
=
\frac{t}{c_i}\frac{|\cD_{h,i}|}{M_h}+o(1),
\]
and hence
\[
|\cD_{h,i}|=\frac{c_i}{L}M_h+o(M_h).
\]

Finally, put $\cC_{h,i}:=\cD_{h,i}\setminus\cR_h$. Since
$|\cR_h|=o(M_h)$, the sets $\cC_{h,i}$ are pairwise disjoint,
\[
\left|[M_h]\setminus\bigcup_{i=1}^r\cC_{h,i}\right|=o(M_h),
\qquad
|\cC_{h,i}|=\frac{c_i}{L}M_h+o(M_h).
\]
For $u\in\cC_{h,i}$, no replacement was made, so
$\|Y'_{h,u}-Y_{h,u}\|_\infty\le\delta_h$. Thus
\[
\left\|Y_{h,u}-\frac{t}{c_i}e_i\right\|_\infty
\le\tau_h+\delta_h.
\]
After passing to a further subsequence, choose $\xi_h\downarrow0$ with
$\xi_h\ge\tau_h+\delta_h$. This proves the lemma.
\end{proof}

The next lemma converts the preceding vector structure into canonical edge
profiles. Its proof is the only use of regularity in the robust theorem.

\begin{lemma}\label{lem:critical-canonical}
Let $r\ge2$, let $m_h\to\infty$, and put $n_h=(t+r-1)m_h+O(1)$. Let
$\cG_h$ be a $t$-graph on $n_h$ vertices with
$\Delta(\overline{\cG_h})=o(m_h^{t-1})$, and let $\chi_h$ be an
$r$-edge-colouring satisfying, uniformly in $i\in[r]$,
$\nu_i(\chi_h)\le m_h+O(1)$. After passing to a subsequence and permuting
the colours, there are fixed positive integers $c_1,\ldots,c_r$ satisfying
$c_1+\cdots+c_r=t+r-1$ and pairwise disjoint sets
$A_{h,1},\ldots,A_{h,r}$ such that $|A_{h,i}|=c_im_h+o(m_h)$.
Moreover, if $p^{(i)}$ is defined by \eqref{eq:canonical-profile}, then
the family of profile-$p^{(i)}$ sets that are either missing from $\cG_h$
or are not coloured $i$ has size $o(m_h^t)$.
\end{lemma}

\begin{proof}
Choose $d_h\downarrow0$, $\varepsilon_h=o(d_h)$, and $M_{0,h}\to\infty$.
For each $h$, choose parameters
\[
0<\eta_h\ll1/Z_h\ll1/M_{0,h}\ll\varepsilon_h\ll d_h/r,1/r,1/t,
\]
and let $n_{0,h}$ be large enough for Lemmas~\ref{hreg} and
\ref{lem:chitoskeleton_hygr}. Passing recursively to a subsequence, assume
that
\[
n_h\ge\max\{n_{0,h},hZ_h\}
\quad\text{and}\quad
\frac{\Delta(\overline{\cG_h})}{n_h^{t-1}}
\le(1-d_h)\eta_h^{t-1}.
\]
If $X_1,\ldots,X_t$ are pairwise disjoint and $|X_j|\ge\eta_hn_h$, then
\[
e_{\overline{\cG_h}}(X_1,\ldots,X_t)
\le |X_1|\Delta(\overline{\cG_h})
\le(1-d_h)\prod_{j=1}^t|X_j|.
\]
Thus $\cG_h$ is $(\eta_h,1,d_h,2)$-uniform.

Apply Lemma~\ref{hreg}, followed by
Lemma~\ref{lem:chitoskeleton_hygr} and
Observation~\ref{lem:initialvalues}. We obtain an equitable partition into
$M_h$ clusters, where $M_{0,h}\le M_h\le Z_h$, and vectors
$S_{h,u}\in[0,t]^r$. Since $M_{0,h}\to\infty$ and
$M_h/n_h\le Z_h/n_h\le1/h$, we have $M_h\to\infty$ and $M_h=o(n_h)$.
The vectors satisfy the almost-covering hypothesis of
Lemma~\ref{lem:critical-rigidity}. If $k_h=\lfloor n_h/M_h\rfloor$, then
Lemma~\ref{lem:chitoskeleton_hygr} and
Observation~\ref{lem:initialvalues} give
\[
\frac1{M_h}\sum_{u=1}^{M_h}s_{h,u}^{(i)}
\le
\frac{t\bigl(m_h+O(1)+3\varepsilon_hn_h\bigr)}{k_hM_h}
=
\frac{t}{t+r-1}+o(1),
\]
uniformly in $i$. Lemma~\ref{lem:critical-rigidity} therefore gives the
integers $c_i$ and cluster classes $\cC_{h,i}$. Let $A_{h,i}$ be the union
of the clusters indexed by $\cC_{h,i}$. Equitability and $M_h=o(n_h)$ give
$|A_{h,i}|=c_im_h+o(m_h)$.

Fix $i$, and consider a profile-$p^{(i)}$ set using distinct clusters from
the classes $\cC_{h,j}$. Suppose its cluster tuple is regular in every
colour. If a wrong colour $j\ne i$ had density at least $d_h/r$, then the
dual constraint and the conclusion of Lemma~\ref{lem:critical-rigidity}
would give
\[
t\le\sum_{u\in Q}s_{h,u}^{(j)}
\le(c_j-1)\frac{t}{c_j}+o(1)<t,
\]
a contradiction. Hence the total wrong-colour density on every such tuple
is at most $d_h$. The regular tuples contribute $O_{r,t}(d_hn_h^t)$ wrong
sets, the irregular tuples contribute $O_{r,t}(\varepsilon_hn_h^t)$, and
sets using a cluster twice contribute $O_t(n_h^t/M_h)$. Finally,
\[
e(\overline{\cG_h})
\le\frac{n_h}{t}\Delta(\overline{\cG_h})
=o(m_h^t).
\]
All four terms are $o(m_h^t)$.
\end{proof}

We also need an exact switching statement. Let
$\mathbf a=(a_1,\ldots,a_s)$ be a non-negative integer vector with
$\sum_ja_j=t$. For disjoint sets $W_1,\ldots,W_s$ with $|W_j|=a_jq$,
let $\cK(\mathbf a;q)$ be the family of all $t$-sets $e$ with
$|e\cap W_j|=a_j$ for every $j$. For
$\cB\subseteq\cK(\mathbf a;q)$, write
$d_{\cB}(v)=|\{e\in\cB:v\in e\}|$ and
$\Delta(\cB)=\max_v d_{\cB}(v)$.

\begin{lemma}\label{lem:profile-switching}
If $q\to\infty$ and $\Delta(\cB)=o(q^{t-1})$, then
$\cK(\mathbf a;q)\setminus\cB$ contains a matching of size $q$ for all
sufficiently large $q$.
\end{lemma}

\begin{proof}
Partition each $W_j$ into $q$ labelled sets of size $a_j$ and unite sets
with the same label. This gives a matching of size $q$. Choose such a
matching $\cM$ with the fewest bad edges, and suppose that
$e_0\in\cM\cap\cB$.

Label the $t$ profile positions by $\mathbb Z_t$ and order the vertices in
every edge of $\cM$ accordingly. Choose an ordered $(t-1)$-tuple of
distinct edges $e_1,\ldots,e_{t-1}\in\cM\setminus\{e_0\}$. If $v_{a,b}$
is the vertex in position $b$ of $e_a$, put
$f_s=\{v_{a,b}:a+b=s\text{ in }\mathbb Z_t\}$. The edges
$f_0,\ldots,f_{t-1}$ are disjoint, have profile $\mathbf a$, and cover the
same vertices as $e_0,\ldots,e_{t-1}$.

Let $x_s$ be the unique vertex of $f_s\cap e_0$. A prescribed bad edge
$f\ni x_s$ is produced as $f_s$ by at most one ordered choice: each other
vertex of $f$ determines both its old matching edge and the row occupied by
that edge. Thus at most $d_{\cB}(x_s)$ choices make $f_s$ bad, and at most
$t\Delta(\cB)$ choices make some $f_s$ bad. The total number of ordered
choices is $(q-1)(q-2)\cdots(q-t+1)$. For large $q$ this exceeds
$t\Delta(\cB)$, so one switching decreases the number of bad edges,
contradicting the choice of $\cM$.
\end{proof}

We first prove the critical-order form of
Theorem~\ref{thm:robust-critical-AFL}.

\begin{lemma}\label{lem:critical-threshold}
Fix $r\ge1$ and $t\ge2$, and put $L=t+r-1$. Let $\cG_m$ be a $t$-graph
on $Lm-r+1$ vertices such that
$\Delta(\overline{\cG_m})=o(m^{t-1})$. Then, for all sufficiently large
$m$, every $r$-edge-colouring of $\cG_m$ contains a monochromatic matching
of size $m$.
\end{lemma}

\begin{proof}
If $r=1$, apply Lemma~\ref{lem:profile-switching} with one part of size
$tm$, profile $(t)$, and bad family $E(\overline{\cG_m})$. Hence assume
$r\ge2$.

Suppose the assertion fails along a sequence $m\to\infty$. Let $\chi$ be
an $r$-edge-colouring of $\cG=\cG_m$ with
$\nu_i(\chi)\le m-1$ for every $i$. Lemma~\ref{lem:critical-canonical}
gives fixed positive integers $c_1+\cdots+c_r=L$, disjoint sets
$A_1,\ldots,A_r$ with $|A_i|=c_im+o(m)$, and bad families $\cB_i$ of
size $o(m^t)$ for the profiles $p^{(i)}$ in
\eqref{eq:canonical-profile}. Here a profile edge is bad if it is missing
from $\cG$ or is not coloured $i$.

Put $\cB=\bigcup_i\cB_i$ and write $|\cB|=\beta m^t$, where
$\beta=o(1)$. Set $\theta=\sqrt\beta+m^{-1/2}$. Let $X$ contain all
vertices outside $A_1\cup\cdots\cup A_r$ and all vertices whose degree in
$\cB$ exceeds $\theta m^{t-1}$. Since
$\sum_v d_{\cB}(v)=t|\cB|$, we have $|X|=o(m)$. Put
$V_i=A_i\setminus X$ and write $|V_i|=c_im-d_i$, where $d_i=o(m)$. As
$X,V_1,\ldots,V_r$ partition the $Lm-r+1$ vertices,
\begin{equation}\label{eq:defect-identity}
\sum_{i=1}^r d_i=|X|+r-1.
\end{equation}
Every vertex outside $X$ has degree at most $\theta m^{t-1}$ in each
$\cB_i$.

For every $x\in X$, choose pairwise disjoint sets
$S_x\subseteq V_1\cup\cdots\cup V_r$ such that
$|S_x\cap V_j|=c_j-1$ for every $j$ and
$e_x:=\{x\}\cup S_x\in E(\cG)$. These are $t$-sets since
$1+\sum_j(c_j-1)=t$. They can be chosen greedily and pairwise disjoint.
After $o(m)$ previous choices every relevant part still has linear size,
so there are $\Omega(m^{t-1})$ candidates, while at most
$d_{\overline{\cG}}(x)=o(m^{t-1})$ are missing.

Let $k_i=|\{x\in X:\chi(e_x)=i\}|$. Then $\sum_i k_i=|X|$, and
\eqref{eq:defect-identity} implies that $k_i\ge d_i$ for some $i$.
Otherwise, integrality gives
$|X|=\sum_i k_i\le\sum_i(d_i-1)=|X|-1$. Fix this $i$, retain the $k_i$
colour-$i$ edges $e_x$, and put $q=m-k_i$. Then $q=(1-o(1))m$.

After deleting the core vertices used by these edges, at least
\[
c_im-d_i-k_i(c_i-1)
=
c_iq+(k_i-d_i)
\ge c_iq
\]
vertices remain in $V_i$, while for $j\ne i$ at least
\[
c_jm-d_j-k_i(c_j-1)
=
(c_j-1)q+(m-d_j)
\ge(c_j-1)q
\]
vertices remain in $V_j$. Choose subsets of these exact sizes. In the
resulting profile-$p^{(i)}$ hypergraph, the bad family is contained in
$\cB_i$ and has maximum degree at most
$\theta m^{t-1}=o(q^{t-1})$. Lemma~\ref{lem:profile-switching} gives $q$
disjoint colour-$i$ edges of $\cG$. Together with the retained edges they
form a colour-$i$ matching of size $m$, a contradiction.
\end{proof}

\begin{proof}[Proof of Theorem~\ref{thm:robust-critical-AFL}]
Put $L=t+r-1$. Suppose, for a contradiction, that the equivalent complement
formulation above fails. For every $h\in\mathbb N$, choose $n_h$
sufficiently large, a $t$-graph $\cG_h$ on $n_h$ vertices, and an
$r$-edge-colouring of $\cG_h$ with no monochromatic matching of size
$m_h:=\lfloor(n_h+r-1)/L\rfloor$, such that
$\Delta(\overline{\cG_h})\le h^{-1}n_h^{t-1}$. We may choose the $n_h$
recursively so that $m_h$ is strictly increasing.

Set $N_h=Lm_h-r+1$ and choose any $N_h$ vertices of $\cG_h$. Let $\cH_h$
be the induced $t$-graph with the inherited colouring. Since $n_h-N_h<L$,
we have $m_h=\Theta(n_h)$, and hence
\[
\Delta(\overline{\cH_h})
\le\Delta(\overline{\cG_h})
\le h^{-1}n_h^{t-1}
=o(m_h^{t-1}).
\]
The inherited colouring of $\cH_h$ has no monochromatic matching of size
$m_h$.

For $m\notin\{m_h:h\in\mathbb N\}$, let
$\cH_m=K_{Lm-r+1}^{(t)}$. Then
$\Delta(\overline{\cH_m})=o(m^{t-1})$, while for infinitely many
$m=m_h$ the graph $\cH_m$ admits an $r$-edge-colouring with no
monochromatic matching of size $m$. This contradicts
Lemma~\ref{lem:critical-threshold}.
\end{proof}

\subsection{Stable Kneser hypergraphs}
\label{subsec:critical-corollaries}

We prove Theorem~\ref{thm:stable}, and hence the value predicted by
Meunier's conjecture throughout the stated linear regime. Regard $[n]$ as cyclically ordered and call a $t$-set $s$-stable if every
two of its elements have cyclic distance at least $s$. Let
$\cS_{n,t,s}$ be the $t$-graph of $s$-stable $t$-sets.

\begin{proof}[Proof of Theorem~\ref{thm:stable}]
Let $r_0=\lceil2/\gamma\rceil$. For each $j\in[r_0]$, let
$c_j>0$ and $n_j$ be supplied by the equivalent complement
formulation of Theorem~\ref{thm:robust-critical-AFL} with $r=j$, and put
$
c_*:=\min_{j\in[r_0]}c_j$ and $
n_*:=\max_{j\in[r_0]}n_j.
$
Choose a constant $C_t>0$ such that
\[
2(s-1)\binom{n-2}{t-2}
+n(s-1)\binom{n-3}{t-3}
\le C_tsn^{t-2},
\]
where the second term is omitted when $t=2$, and choose
$c=c(t,\gamma)>0$ sufficiently small that
$
c\le1
\text{ and } \frac{C_tc}{t}\le c_*.
$

Set
$R=\left\lceil\frac{n-q(t-1)}{q-1}\right\rceil$.
For the upper bound, $(R-1)(q-1)<n$, so choose disjoint sets
$B_1,\ldots,B_{R-1}$ of size $q-1$, and let $B_R$ contain the remaining
vertices. The definition of
$R$ gives $|B_R|\le tq-1$. Colour an $s$-stable $t$-set by the least
$i<R$ that it meets, if such an $i$ exists, and by $R$ otherwise. A
colour-$i$ matching with $i<R$ has size at most $q-1$, while a colour-$R$
matching of size $q$ would require $tq$ vertices in $B_R$. Thus the
chromatic number is at most $R$.

For the lower bound, suppose that a proper colouring with $R-1$ colours
exists, and put $r:=R-1$. Since $q\ge\gamma n$, for all sufficiently large
$n$,
\[
r<
\frac{n-q(t-1)}{q-1}
\le\frac{n}{q-1}
\le\frac{2}{\gamma},
\]
so $r\le r_0$. Such a colouring is an $r$-edge-colouring of
$\cS_{n,t,s}$ with no monochromatic matching of size $q$.

We have
\[
\Delta(\overline{\cS_{n,t,s}})=O_t(sn^{t-2}).
\]
Indeed, for a fixed vertex $v$, a non-stable $t$-set containing $v$ either
contains a point within cyclic distance $s-1$ of $v$, or contains a close
pair not involving $v$. These cases contribute at most
\[
2(s-1)\binom{n-2}{t-2}
+n(s-1)\binom{n-3}{t-3},
\]
where the second term is omitted when $t=2$. Since
$s\le cq\le cn/t$,
\[
\Delta(\overline{\cS_{n,t,s}})
\le C_tsn^{t-2}
\le\frac{C_tc}{t}n^{t-1}
\le c_*n^{t-1}.
\]
For $n\ge n_*$, Theorem~\ref{thm:robust-critical-AFL} therefore gives a
monochromatic matching of size at least
$\lfloor(n+r-1)/(t+r-1)\rfloor$.

Finally, $r<\big(n-q(t-1)\big)/(q-1)$,
so $n+r>q(t+r-1)$. Both sides are integers, and hence
$n+r-1\ge q(t+r-1)$. The matching supplied by the robust theorem has size
at least $q$, a contradiction.
\end{proof}

\subsection{The transference AFL theorem}
\label{subsec:transAFL}

The robust theorem concerns almost-complete hosts. For sparse random hosts we use Theorem~\ref{thm:hypergraphstructure}
together with the strip-colouring bound in Lemma~\ref{lem:aaalower}. More precisely, we recover the following transference version of the AFL
theorem, first proved in~\cite{GGMS}.

\begin{theorem}[Gishboliner, Glock, Michaeli and Sgueglia]
\label{thm:randomAFL}
Let $t,r\ge2$ and $\varepsilon>0$, and let
$G\sim G^{(t)}(n,p)$. There exists $C>0$ such that, if
$p\ge Cn^{-t+1}$, then with high probability every $r$-edge-colouring
$\chi$ of $G$ satisfies
\[
\nu(\chi)\ge(1-\varepsilon)\frac{n}{r+t-1}.
\]
\end{theorem}

We use the following uniformity estimate from~\cite{GGMS}.

\begin{lemma}\label{lem:asdf}
Let $1/C\ll\eta,1/t$, and let $G\sim G^{(t)}(n,p)$ with
$p\ge Cn^{-t+1}$. Then with high probability, for every collection of
pairwise disjoint sets $X_1,\ldots,X_t$ with $|X_i|\ge\eta n$,
\[
\frac p2
\le
\frac{e(X_1,\ldots,X_t)}{|X_1|\cdots|X_t|}
\le
\frac{3p}{2}.
\]
\end{lemma}

\begin{proof}[Proof of Theorem~\ref{thm:randomaaa0}]
Fix $\varepsilon>0$, and let $\eta>0$ be supplied by
Theorem~\ref{thm:hypergraphstructure} with $d=1/2$ and $D=3/2$. Choose
$C$ sufficiently large for Lemma~\ref{lem:asdf}. On the event in that
lemma, $G$ is $(\eta,p,1/2,3/2)$-uniform. Theorem~\ref{thm:hypergraphstructure}
therefore gives an $r$-strip colouring $\psi$ of $K_n^{(t)}$ satisfying
$\nu(\chi)\ge\nu(\psi)-\varepsilon n$.
\end{proof}

\begin{proof}[Proof of Theorem~\ref{thm:randomAFL}]
Apply Theorem~\ref{thm:randomaaa0} with
$\delta=\varepsilon/[2(r+t-1)]$. With high probability, every colouring
$\chi$ of $G$ admits an $r$-strip colouring $\psi$ such that
$\nu(\chi)\ge\nu(\psi)-\delta n$. By Lemma~\ref{lem:aaalower},
\[
\nu(\chi)
\ge
\left\lfloor\frac{n}{r+t-1}\right\rfloor-\delta n
\ge
(1-\varepsilon)\frac{n}{r+t-1}
\]
for all sufficiently large $n$.
\end{proof}

\section{Structural reduction for graph tilings}
\label{sec:graphstructure}

In this section, we prove Theorem~\ref{thm:graphstructure}. Throughout, \(H\) is a fixed graph with at least one edge and we write $h:=v(H)$.
\footnote{We believe that the argument can also be adapted to the setting in
which different colours have different target graphs, that is, to
Ramsey-tiling problems of the form $R(kH_1,kH_2,\ldots,kH_r)$. Indeed, the
method treats the colours separately: for each colour one considers the
corresponding family of reduced $H_i$-patterns and the associated dual linear
program. For the sake of a cleaner presentation, however, we restrict
throughout to a common target graph $H$ and do not pursue this extension here.}

\subsection{Preliminaries: reduced \(H\)-patterns}
We use the following standard consequence of the K{\L}R embedding theorem;
see~\cite{CGSS}. The only minor difference from the usual formulation is
that we allow \(H\) to have \(K_2\)-components and isolated vertices; these
components can be embedded greedily after applying the K{\L}R theorem to the
remaining part of \(H\).

\begin{lemma}[Partite K{\L}R embedding]
\label{lem:HpatternKLR}
Let \(H\) be a fixed graph with at least one edge. For every \(d>0\),
there exists \(\varepsilon_{\mathrm{KLR}}>0\) such that, for every
\(\eta>0\), there exists \(C_1>0\) with the following property.
If \(G\sim G(n,p)\) with
$p\ge C_1n^{-1/m_2^*(H)},$
then with high probability the following holds.

Let \(F\subseteq G\), and let
\(\{U_x:x\in V(H)\}\) be pairwise disjoint sets satisfying
$|U_x|\ge\eta n$
for every $x\in V(H).$
If, for every edge \(xy\in E(H)\), the pair \(F[U_x,U_y]\) is
\((\varepsilon_{\mathrm{KLR}},p)\)-regular
and has density at least \(dp\), then \(F\) contains a canonical copy of
\(H\), with each \(x\in V(H)\) embedded in \(U_x\).
\end{lemma}\begin{proof}
Let \(H_{\ge3}\) be the union of all components of \(H\) having at least
three vertices, and let \(H_2\) be the union of all \(K_2\)-components of
\(H\). The isolated vertices impose no edge constraints.

We first embed the components of \(H_{\ge3}\), one by one. For each such
component \(J\), the usual partite K{\L}R embedding theorem applies at the
threshold \(p\ge C_1n^{-1/m_2^*(H)}\), since \(m_2(J)\le m_2^*(H)\).
After embedding one component, only \(O_H(1)\) vertices have been used, and
the vertex classes for the remaining components still have linear size. By
increasing \(C_1\) and decreasing \(\varepsilon_{\mathrm{KLR}}\) if
necessary, the same K{\L}R argument applies successively to all components
of \(H_{\ge3}\).

It remains to embed the \(K_2\)-components and isolated vertices. For each
\(K_2\)-component \(xy\), the pair \(F[U_x,U_y]\) has density at least
\(dp\). Since \(|U_x|,|U_y|\ge\eta n\) and
\(p\ge C_1n^{-1/m_2^*(H)}\ge C_1n^{-1}\), this pair contains an edge for
all sufficiently large \(n\). We choose one such edge for every
\(K_2\)-component. Finally, choose arbitrary vertices from the prescribed
sets for the isolated vertices. Since the sets \(U_x\), \(x\in V(H)\), are
pairwise disjoint, these choices together form a canonical copy of \(H\).
\end{proof}

\smallskip
Let \((V_1,\ldots,V_M)\) be an
\((\varepsilon_{\mathrm{reg}},p)\)-regular
equipartition of an \(r\)-edge-coloured graph \(G\). For each colour
\(\ell\in[r]\), let
$\mathcal X_\ell\subseteq\binom{[M]}2$
be the collection of pairs \(\{i,j\}\) for which
\(G_\ell[V_i,V_j]\) is
\((\varepsilon_{\mathrm{reg}},p)\)-regular
and has density at least \(p/(2r)\).

A map
$\phi:V(H)\to[M]$
is called a \emph{colour-\(\ell\) reduced \(H\)-pattern} if, for every
\(xy\in E(H)\), we have \(\phi(x)\ne\phi(y)\) and
$
\{\phi(x),\phi(y)\}\in\mathcal X_\ell.
$
Non-adjacent vertices of \(H\) may be mapped to the same cluster. We write
\(\Phi_\ell\) for the family of all colour-\(\ell\) reduced \(H\)-patterns.
For $\phi\in\Phi_\ell$ and $i\in[M]$, let $c_i(\phi):=|\phi^{-1}(i)|$ be the number of vertices of $H$ mapped to cluster $i$ by $\phi$.

The next lemma lifts a feasible fractional packing of reduced
\(H\)-patterns to an actual \(H\)-tiling.

\begin{lemma}
\label{lem:embeddinggraph}
Fix \(r\ge2\), \(Z\in\mathbb N\), and \(\gamma>0\). There exist
\(\varepsilon_0>0\) and \(C>0\) such that, for
\(G\sim G(n,p)\) with
$p\ge Cn^{-1/m_2^*(H)},$
the following holds with high probability.

Let \(F\subseteq G\), and let
$V(F)=V_1\cup\cdots\cup V_M,$
$M\le Z,$
be an equipartition, where \(k=\lfloor n/M\rfloor\). Let
\(\mathcal X\subseteq\binom{[M]}2\) be a collection of pairs such that
\(F[V_i,V_j]\) is \((\varepsilon_0,p)\)-regular with density at least
\(p/(2r)\) for every \(\{i,j\}\in\mathcal X\). Let \(\Phi\) be a family
of maps \(\phi:V(H)\to[M]\) such that
$\{\phi(x),\phi(y)\}\in\mathcal X$
for every $\phi\in\Phi$
and every $xy\in E(H).$
If non-negative weights \(\{x_\phi\}_{\phi\in\Phi}\) satisfy
\[
\sum_{\phi\in\Phi}c_i(\phi)x_\phi\le1-\gamma
\qquad\text{for every }i\in[M],
\]
then \(F\) contains an \(H\)-tiling of size at least
\[
k\sum_{\phi\in\Phi}x_\phi-O_{H,Z}(1).
\]
\end{lemma}

\begin{proof}
Apply Lemma~\ref{lem:HpatternKLR} with \(d=1/(4r)\), and let
\(\varepsilon_{\mathrm{KLR}}>0\) be the resulting constant. Choose
\(\varepsilon_0>0\) sufficiently small that every restriction of an
\((\varepsilon_0,p)\)-regular pair to subsets of relative size at least
\(\gamma/(3h)\) is
\((\varepsilon_{\mathrm{KLR}},p)\)-regular and has density at least
\(p/(4r)\). This follows from the slicing lemma. Set
$\eta_0:=\frac{\gamma}{10hZ},$
and condition on the high-probability event in
Lemma~\ref{lem:HpatternKLR} with \(\eta=\eta_0\).

For each \(\phi\in\Phi\), let
$m_\phi:=\lfloor kx_\phi\rfloor.$
We process the patterns one by one and greedily embed \(m_\phi\) canonical
copies of type \(\phi\). The total number of vertices requested from a
cluster \(V_i\) is at most
\[
\sum_{\phi\in\Phi}c_i(\phi)m_\phi
\le
k\sum_{\phi\in\Phi}c_i(\phi)x_\phi
\le
(1-\gamma)k.
\]
Hence, before any prescribed copy is embedded, every cluster contains at
least \(\gamma k\) unused vertices.

Fix a pattern \(\phi\) at one step of the greedy procedure. For each
\(i\in[M]\), split the unused vertices of \(V_i\) into
\(|\phi^{-1}(i)|\) disjoint sets, one for each vertex of \(H\) mapped to
\(i\). Since \(|\phi^{-1}(i)|\le h\), each such set may be chosen to have
size at least
$\frac{\gamma k}{h}-1
\ge
\frac{\gamma k}{2h}
\ge
\eta_0 n$
for all sufficiently large \(n\). Moreover, each such set has relative
size at least \(\gamma/(3h)\) inside its cluster. For every edge
\(xy\in E(H)\), the corresponding pair is therefore
\((\varepsilon_{\mathrm{KLR}},p)\)-regular and has density at least
\(p/(4r)\). Lemma~\ref{lem:HpatternKLR} supplies the required canonical
copy of \(H\). Thus the greedy procedure succeeds.

Finally,
\[
\sum_{\phi\in\Phi}m_\phi
\ge
k\sum_{\phi\in\Phi}x_\phi-|\Phi|
\ge
k\sum_{\phi\in\Phi}x_\phi-O_{H,Z}(1),
\]
because \(|\Phi|\le M^h\le Z^h\).
\end{proof}

We apply Lemma~\ref{lem:embeddinggraph} with
\(F=G_\ell\), \(\mathcal X=\mathcal X_\ell\), and
\(\Phi=\Phi_\ell\).

\subsection{Step I: Reducing random graphs to linear programs}
\label{sec:Skeletonframeworkgraphs}

\begin{lemma}\label{lem:chitoskeleton_gr}
Fix \(H\), \(r\), and \(\xi>0\). There exists
\(\varepsilon_{\mathrm{reg}}>0\) such that, for every \(Z\), there is
\(C>0\) for which the following holds with high probability whenever
\(G\sim G(n,p)\) and
$p\ge Cn^{-1/m_2^*(H)}.$
Let $\chi$ be any $r$-edge-colouring of $G$, and let
$V(G)=V_1\cup\cdots\cup V_M$ be an equipartition with $M\le Z$. Put
$k=\lfloor n/M\rfloor$. Define $\mathcal X_\ell$ and $\Phi_\ell$ using
$(\varepsilon_{\mathrm{reg}},p)$-regular pairs as above. Then, for every
$\ell\in[r]$,
\[
\nu_H(G_\ell)\ge kZ_\ell-\xi n,
\]
where \(Z_\ell\) is the optimal value of the following linear program.

\begin{align}
\tag{LP}\label{eq:LP}
\text{minimize}\quad
& \sum_{i=1}^M z_i,\\
\text{subject to}\quad
& \sum_{x\in V(H)}z_{\phi(x)}\ge1
\qquad\text{for every }\phi\in\Phi_\ell,\notag\\
& z_i\ge0
\qquad\text{for every }i\in[M].\notag
\end{align}
\end{lemma}

\begin{proof}
The primal linear program dual to \eqref{eq:LP} is
\[
\begin{aligned}
\text{maximize}\quad
& \sum_{\phi\in\Phi_\ell}x_\phi,\\
\text{subject to}\quad
& \sum_{\phi\in\Phi_\ell}c_i(\phi)x_\phi\le1
\qquad\text{for every }i\in[M],\\
& x_\phi\ge0
\qquad\text{for every }\phi\in\Phi_\ell.
\end{aligned}
\]
By LP duality, its optimal value is \(Z_\ell\).

Choose \(\gamma>0\) sufficiently small in terms of \(\xi\), and take
\(\varepsilon_{\mathrm{reg}}\) small enough for
Lemma~\ref{lem:embeddinggraph}. Let \(\{x_\phi\}\) be an optimal primal
solution. The scaled solution
\(\{(1-\gamma)x_\phi\}\) satisfies
\[
\sum_{\phi\in\Phi_\ell}
c_i(\phi)(1-\gamma)x_\phi
\le1-\gamma
\qquad (i\in[M]).
\]
Therefore Lemma~\ref{lem:embeddinggraph} gives
\[
\nu_H(G_\ell)
\ge
k(1-\gamma)Z_\ell-O_{H,Z}(1).
\]
The constant vector \(z_i=1\) is feasible in \eqref{eq:LP}, so
\(Z_\ell\le M\). Hence
$\gamma kZ_\ell\le\gamma kM\le\gamma n.$
Taking \(\gamma\ll\xi\) and then \(n\) sufficiently large yields
$\nu_H(G_\ell)\ge kZ_\ell-\xi n.$\end{proof}

\subsection{Step II: Repair, direct compression, and Ramsey extraction}

For each colour $\ell\in[r]$, let
$\{s_i^{(\ell)}\}_{i\in[M]}$ be an optimal solution to
\eqref{eq:LP}, and put
\[
S_i:=\bigl(s_i^{(1)},\ldots,s_i^{(r)}\bigr),
\qquad
Z:=(Z_1,\ldots,Z_r).
\]
We may assume that $S_i\in[0,1]^r$, since replacing a coordinate larger
than $1$ by $1$ preserves every constraint. In particular,
$\sum_iS_i=Z$.

Let $R_0=R_0(r,h)$ be sufficiently large that the following multipartite
Ramsey statement holds. If $q\le r$ and $B_1,\ldots,B_q$ are disjoint sets
of size $R_0$, then every $r$-colouring of the pairs in their union contains
sets $A_j\subseteq B_j$, $|A_j|=h$, such that all pairs inside each $A_i$
and all pairs between each $A_i,A_j$ have fixed colours.

\begin{lemma}\label{lem:vertorrefinegraph}
For every $\xi>0$, there are $m_0\in\mathbb N$ and $\zeta>0$ such that the
following holds whenever $M\ge m_0$ and
\[
\left|\binom{[M]}2\setminus\bigcup_{\ell\in[r]}\mathcal X_\ell\right|
\le\zeta M^2.
\]
There exist $q\le r$, distinct vectors $R_1,\ldots,R_q\in[0,2]^r$,
positive numbers $\alpha_1,\ldots,\alpha_q$ summing to one, colours
$\lambda(i,j)\in[r]$ for $1\le i\le j\le q$, pairwise disjoint sets
$A_1,\ldots,A_q\subseteq[M]$ of size $h$, and the $q$-strip colouring $\psi$
of $K_n$ with template $\lambda$ and part sizes $\alpha_jn+O(1)$ such that:
\begin{enumerate}
\item[(G1)]
$\nu_H(\psi)\le \frac nM\|Z\|_\infty+\xi n+2r.$
\item[(G2)] For every $1\le i\le j\le q$, every pair of distinct indices
inside $A_i$ when $i=j$, and every pair with one endpoint in $A_i$ and the
other in $A_j$ when $i<j$, belongs to
$\mathcal X_{\lambda(i,j)}$.
\item[(G3)] If $\ell\in[r]$ and $f:V(H)\to[q]$ satisfies
$\lambda(f(x),f(y))=\ell$ for every $xy\in E(H)$, then
\[
\sum_{x\in V(H)}R_{f(x)}^{(\ell)}\ge1.
\]
\end{enumerate}
\end{lemma}

\begin{proof}
Choose $\delta,\theta>0$ so that, with
$N_0=(\lceil1/\delta\rceil+2)^r$, we have
$N_0\theta<1$ and $\delta+2N_0\theta<\xi$. Choose
$m_0\ge r$ so that $\theta m_0\ge2R_0$, and then choose
$\zeta<\theta^2/(4R_0^2)$.

Round every coordinate of $S_i$ upwards to the nearest multiple of
$\delta$, obtaining $\widetilde S_i\in[0,2]^r$. There are at most $N_0$
rounded values. Call a value frequent if it occurs at least $\theta M$
times. Since $N_0\theta<1$, some frequent value exists. Replace every
non-frequent value by an arbitrary frequent one, and denote the resulting
vectors by $S'_1,\ldots,S'_M$. Fewer than $N_0\theta M$ indices are changed,
so
\[
\left\|\sum_iS'_i\right\|_\infty
\le \|Z\|_\infty+(\delta+2N_0\theta)M.
\]

Let $P=\operatorname{conv}\{S'_1,\ldots,S'_M\}$ and
$\overline S=M^{-1}\sum_iS'_i$. Move from $\overline S$ in the negative
all-ones direction as far as possible while remaining in $P$, and call the
resulting point $J$. Then $J\in\partial P$ and
$\|J\|_\infty\le\|\overline S\|_\infty$. By
Theorem~\ref{thm:Cara}, after discarding zero coefficients, there are
vertices $R_1,\ldots,R_q$ of $P$, where $q\le r$, and positive numbers
$\alpha_1,\ldots,\alpha_q$ summing to one such that
$J=\sum_j\alpha_jR_j$. Hence
\[
\|J\|_\infty
\le \frac1M\|Z\|_\infty+\delta+2N_0\theta.
\]

Every $R_j$ is a frequent rounded value. Let
$T_j=\{u\in[M]:\widetilde S_u=R_j\}$. The sets $T_j$ are pairwise
disjoint, $|T_j|\ge\theta M$, and $S_u\le R_j$ coordinatewise for every
$u\in T_j$.

Put
$\mathcal B=\binom{[M]}2\setminus\bigcup_{\ell\in[r]}\mathcal X_\ell$.
Choose independently and uniformly an $R_0$-set $B_j\subseteq T_j$. For a
fixed bad pair, the probability that both endpoints are selected is at most
$2R_0^2/(\theta^2M^2)$. Thus the expected number of selected bad pairs is
at most
\[
|\mathcal B|\frac{2R_0^2}{\theta^2M^2}
\le\frac{2\zeta R_0^2}{\theta^2}<1.
\]
We may therefore choose the sets $B_j$ so that no bad pair occurs inside or
between them. Colour every pair in their union by one colour $c(u,v)$ for
which $\{u,v\}\in\mathcal X_{c(u,v)}$. The choice of $R_0$ gives sets
$A_j\subseteq B_j$, $|A_j|=h$, and colours $\lambda(i,j)$ satisfying
\textnormal{(G2)}.

To prove \textnormal{(G3)}, let $f:V(H)\to[q]$ be monochromatic of colour
$\ell$ in the template. Choose distinct indices
$\phi(x)\in A_{f(x)}$ for $x\in V(H)$. Then $\phi\in\Phi_\ell$, and the
dual constraint gives
$\sum_xs_{\phi(x)}^{(\ell)}\ge1$. Since
$\phi(x)\in T_{f(x)}$, we have $S_{\phi(x)}\le R_{f(x)}$, proving
\textnormal{(G3)}.

Finally, partition $V(K_n)$ into $U_1,\ldots,U_q$ so that
$||U_j|-\alpha_jn|\le1$, and colour the pairs inside and between the parts
according to $\lambda$. Let $\mathcal M$ be a colour-$\ell$ $H$-tiling in
the resulting strip colouring. Property \textnormal{(G3)}, applied to every
copy in $\mathcal M$, gives
\[
|\mathcal M|
\le\sum_{j=1}^q|U_j|R_j^{(\ell)}
\le nJ_\ell+2r.
\]
Taking the maximum over $\ell$ and using the bound on $J$ proves
\textnormal{(G1)}.
\end{proof}

\begin{proof}[Proof of Theorem~\ref{thm:graphstructure}]
Fix $\varepsilon>0$ and put $\xi=\varepsilon/4$. Apply
Lemma~\ref{lem:vertorrefinegraph} with error $\xi$, obtaining $m_0$ and
$\zeta$. Choose $\varepsilon_{\mathrm{reg}}>0$ so that
$\varepsilon_{\mathrm{reg}}\le\zeta$ and
Lemma~\ref{lem:chitoskeleton_gr} applies with error $\xi$. Apply
Lemma~\ref{hreg}, with $t=2$, $d=1/2$, $D=3/2$, lower bound $m_0$, and
regularity parameter $\varepsilon_{\mathrm{reg}}$. Let $\eta>0$ and $M_1$
be the resulting constants.

Choose $C$ sufficiently large that, with high probability, $G\sim G(n,p)$
is $(\eta,p,1/2,3/2)$-uniform and the universal embedding event in
Lemma~\ref{lem:chitoskeleton_gr} holds for every partition with at most
$M_1$ clusters. Condition on these events and let $\chi$ be an arbitrary
$r$-edge-colouring of $G$. Lemma~\ref{hreg} gives an equipartition into
$m_0\le M\le M_1$ clusters such that all but at most
$\varepsilon_{\mathrm{reg}}M^2$ pairs are regular in every colour. Every
non-exceptional pair has total density at least $p/2$, so it belongs to
$\mathcal X_\ell$ for some $\ell\in[r]$. Hence the hypothesis of
Lemma~\ref{lem:vertorrefinegraph} holds.

Let $Z=(Z_1,\ldots,Z_r)$ be the vector of LP optima. By
Lemma~\ref{lem:chitoskeleton_gr},
$\nu_H(\chi)\ge k\|Z\|_\infty-\xi n.$
Lemma~\ref{lem:vertorrefinegraph} gives an $r$-strip colouring $\psi$ with
\[
\nu_H(\psi)
\le\frac nM\|Z\|_\infty+\xi n+2r
\le k\|Z\|_\infty+\xi n+M+2r,
\]
since $0\le n/M-k<1$ and $\|Z\|_\infty\le M$. Therefore
$\nu_H(\chi)\ge\nu_H(\psi)-2\xi n-M-2r$. For sufficiently large $n$,
$M_1+2r\le\xi n$, and the right-hand side is at least
$\nu_H(\psi)-\varepsilon n$.
\end{proof}

\section{Ramsey graph tilings: finite optimization and explicit values}
\label{sec:graphapp}

Recall that
\[
Rt_r(H;G):=\min_{\chi:E(G)\to[r]}\nu_H(\chi).
\]
For an integer $m\ge1$, let $R_r(mH)$ be the least $n$ such that every
$r$-edge-colouring of $K_n$ contains a monochromatic copy of $mH$. Thus
\[
R_r(mH)=\min\{n:Rt_r(H;K_n)\ge m\}.
\]
We first prove Theorem~\ref{thm:finite-optimization}, which reduces the
complete-host problem to finitely many strip-template linear programs, and
then determine the resulting constant for several graph families.

\subsection{Finite optimization for the asymptotic constant}
\label{subsec:finite-optimization}

For complete hosts, the following statement is a direct consequence of
Theorem~\ref{thm:graphstructure}.

\begin{lemma}
\label{lem:complete-host-graphstructure}
Let $r\ge2$, let $H$ be any fixed graph with at least one edge, and let
$\varepsilon>0$. For all sufficiently large $n$, every
$r$-edge-colouring $\chi$ of $K_n$ admits an $r$-strip colouring $\psi$ of
$K_n$ such that
\[
\nu_H(\chi)\ge\nu_H(\psi)-\varepsilon n.
\]
\end{lemma}

\begin{proof}
Apply Theorem~\ref{thm:graphstructure} with $p=1$, so that
$G(n,1)=K_n$ deterministically.
\end{proof}

Fix $q\ge1$. An $r$-colour $q$-strip template is a map
\[
\lambda:\{(i,j):1\le i\le j\le q\}\to[r].
\]
For $\ell\in[r]$, let $\Phi_\ell(\lambda)$ be the family of maps
$f:V(H)\to[q]$ satisfying
$\lambda(f(u),f(v))=\ell$ for every $uv\in E(H)$, and put
$a_i(f):=|f^{-1}(i)|$. For
$x=(x_1,\ldots,x_q)\in\mathbb R_{\ge0}^q$, let
$\tau_\ell(\lambda,x)$ be the optimum of
\[
\begin{aligned}
\text{maximize}\quad &\sum_{f\in\Phi_\ell(\lambda)}w_f,\\
\text{subject to}\quad
&\sum_{f\in\Phi_\ell(\lambda)}a_i(f)w_f\le x_i
\qquad(i\in[q]),\\
&w_f\ge0\qquad(f\in\Phi_\ell(\lambda)).
\end{aligned}
\]
The function $\tau_\ell(\lambda,\cdot)$ is positively homogeneous. Write
\[
\Delta_q:=\left\{x\in\mathbb R_{\ge0}^q:\sum_i x_i=1\right\}
\]
and define
\[
\beta_{r,H}:=
\min_{1\le q\le r}\ \min_\lambda\ \min_{x\in\Delta_q}
\ \max_{\ell\in[r]}\tau_\ell(\lambda,x).
\]
For fixed $q,\lambda,\ell$, LP duality expresses
$\tau_\ell(\lambda,x)$ as the minimum of $x\cdot z$ over a fixed compact
polytope in $[0,1]^q$. Hence $\tau_\ell(\lambda,\cdot)$ is continuous on
$\Delta_q$, so the minimum above is attained. Moreover, if
$x_i=\max_jx_j$ and $\ell=\lambda(i,i)$, then the pattern mapping every
vertex of $H$ to $i$ gives
\[
\tau_\ell(\lambda,x)\ge\frac{x_i}{v(H)}
\ge\frac{1}{rv(H)}.
\]
Thus $\beta_{r,H}\ge1/[rv(H)]>0$.

If the parts of an integral blow-up of $\lambda$ have sizes
$n_1,\ldots,n_q$, put $N:=\sum_i n_i$ and $x_i:=n_i/N$, and write
$\nu_H^{(\ell)}(\lambda;n_1,\ldots,n_q)$ for its maximum
colour-$\ell$ tiling size. Then
\begin{equation}\label{eq:integral-strip-tiling}
\nu_H^{(\ell)}(\lambda;n_1,\ldots,n_q)
=N\tau_\ell(\lambda,x)+O_{q,H}(1).
\end{equation}
Indeed, every integral tiling gives a feasible fractional solution after
division by $N$. Conversely, taking $\lfloor Nw_f\rfloor$ copies of every
pattern in an optimal solution respects all part capacities and loses only
$O_{q,H}(1)$ copies.

\begin{proof}[Proof of Theorem~\ref{thm:finite-optimization}]
Choosing a minimizing template and a minimizing vector, and rounding the
part sizes, gives
$Rt_r(H;K_n)\le(\beta_{r,H}+o(1))n$ by
\eqref{eq:integral-strip-tiling}. Conversely, let $\chi$ be an arbitrary
$r$-colouring of $K_n$. Lemma~\ref{lem:complete-host-graphstructure},
with an error tending to zero, gives an $r$-strip colouring $\psi$ such that
$\nu_H(\chi)\ge\nu_H(\psi)-o(n)$. If $\lambda$ and $x$ are the template
and part proportions of $\psi$, then
\[
\nu_H(\psi)
=n\max_{\ell\in[r]}\tau_\ell(\lambda,x)+O_{r,H}(1)
\ge\beta_{r,H}n-O_{r,H}(1).
\]
Hence $Rt_r(H;K_n)=(\beta_{r,H}+o(1))n$.

It remains to verify finite computability. For fixed
$(q,\lambda,\ell)$, LP duality gives
\[
\tau_\ell(\lambda,x)=\min_{z\in P_\ell}x\cdot z,
\qquad
P_\ell:=\left\{z\in[0,1]^q:
\sum_i a_i(f)z_i\ge1\text{ for every }f\in\Phi_\ell(\lambda)
\right\}.
\]
The upper bounds $z_i\le1$ are harmless. Let $E_\ell$ be the finite set
of vertices of $P_\ell$. For every choice
$(z_1,\ldots,z_r)\in E_1\times\cdots\times E_r$, restrict to the rational
polytope of $x\in\Delta_q$ on which each $z_\ell$ is optimal, and minimize
$y$ subject to $y\ge x\cdot z_\ell$ for every $\ell$. Enumerating these
finitely many rational linear programs, and then all $q\le r$ and all
templates $\lambda$, computes $\beta_{r,H}$.

Finally, the relation
$R_r(kH)=\min\{n:Rt_r(H;K_n)\ge k\}$ and the already proved asymptotic
formula imply
\begin{align*}
R_r(kH)=\left(\frac1{\beta_{r,H}}+o(1)\right)k.\tag*{\qedhere}  
\end{align*}
\end{proof}

\subsection{Explicit values for \(Rt_r(H;K_n)\)}
\label{subsec:application:3colour}
We now extract closed formulas from the finite optimization in several
natural cases.

\subsubsection{Non-bipartite tilings}

The following result is due to Gy\'arf\'as, S\'ark\"ozy and
Selkow~\cite[Theorem~1.11]{GSS15}. We include its short proof for completeness.

\begin{proposition}[\cite{GSS15}]
\label{prop:nonbipartite}
Let \(r\ge3\) and let \(H\) be a connected non-bipartite graph. Then
\[
Rt_r(H;K_n)=\frac{n}{r\,v(H)}+O_{r,H}(1).
\]
\end{proposition}

\begin{observation}\label{obs:lownonbipar}
For every \(r\)-edge-colouring \(\chi\) of \(K_n\),
\[
\nu_H(\chi)\ge\frac{n}{r\,v(H)}-O_{r,H}(1).
\]
\end{observation}

\begin{proof}
Greedily remove monochromatic copies of \(H\) until fewer than
\(R_r(H)\) vertices remain. This produces at least
\((n-R_r(H))/v(H)\) disjoint copies in total, so one colour contains at
least \(n/(r\,v(H))-O_{r,H}(1)\) of them.
\end{proof}

\begin{lemma}\label{lem:uppernonbipar}
There exists an \(r\)-edge-colouring \(\chi\) of \(K_n\) such that
\[
\nu_H(\chi)\le\frac{n}{r\,v(H)}+O_{r,H}(1).
\]
\end{lemma}

\begin{proof}
Let \(V(K_n)=V_1\cup\cdots\cup V_r\) be an equipartition. Suppose first that \(r\) is odd. Identify the indices and colours with
\(\mathbb Z_r\). Colour the edges inside \(V_i\) with colour \(i\), and
for \(i\ne j\), colour all edges between \(V_i\) and \(V_j\) with
\((i+j)2^{-1}\in\mathbb Z_r\). For a fixed colour \(c\), the
colour-\(c\) pairs between distinct parts form a matching on
\(\mathbb Z_r\setminus\{c\}\), while no such pair is incident with \(V_c\).
Hence the colour-\(c\) graph is the disjoint union of the clique on \(V_c\)
and complete bipartite graphs. Since \(H\) is connected and non-bipartite,
every colour-\(c\) copy of \(H\) lies inside \(V_c\).

Now suppose that \(r\) is even. Index \(V_1,\ldots,V_{r-1}\) and the first
\(r-1\) colours by \(\mathbb Z_{r-1}\), and denote the final part and colour
by \(\infty\). Use the preceding odd construction on the first \(r-1\)
parts, colour the edges inside \(V_\infty\) with \(\infty\), and colour the
edges between \(V_\infty\) and \(V_i\) with \(i+1\pmod{r-1}\). For a fixed
\(c\in\mathbb Z_{r-1}\), the colour-\(c\) pairs between parts form a
matching together with at most one additional adjacent edge, and hence form
a forest; moreover, \(V_c\) is isolated from all colour-\(c\) cross-pairs.
Thus the colour-\(c\) graph outside \(V_c\) is bipartite, while colour
\(\infty\) occurs only inside \(V_\infty\). Again, every monochromatic copy
of \(H\) lies inside one designated part.

Therefore, every monochromatic \(H\)-tiling has at most
\(\max_i\lfloor |V_i|/v(H)\rfloor=n/(r\,v(H))+O_H(1)\) copies.
\end{proof}

Proposition~\ref{prop:nonbipartite} now follows from
Observation~\ref{obs:lownonbipar} and Lemma~\ref{lem:uppernonbipar}.

\subsubsection{Hall-type bipartite tilings for \(r=3\)}

Let \(H\) be a connected bipartite graph with a fixed bipartition
\(A\cup B\), where \(|A|=a\le b=|B|\). We call \(H\) \emph{Hall-type} if
it contains a matching saturating \(A\). Equivalently, by K\"onig's theorem,
the vertex-cover number of \(H\) is \(a\).

\begin{theorem}\label{thm:biparr=3}
Let \(b\ge a\ge1\), and let \(H\) be a connected Hall-type bipartite graph
with bipartition sizes \(a\) and \(b\). Then
\[
Rt_3(H;K_n)=
\begin{cases}
\displaystyle
\left(\frac{1}{3a+b}+o(1)\right)n,
& b\le3a,\\[2ex]
\displaystyle
\left(\frac{2}{3(a+b)}+o(1)\right)n,
& b\ge3a.
\end{cases}
\]
\end{theorem}

\begin{lemma}\label{lem:upper3colour}
Let \(H\) be as in Theorem~\ref{thm:biparr=3}.
\begin{enumerate}
\item If \(b\ge3a\), then there exists a \(3\)-colouring \(\chi\) of
\(K_n\) such that
\[
\nu_H(\chi)\le\frac{2}{3(a+b)}n+O_{a,b}(1).
\]

\item If \(b\le3a\), then there exists a \(3\)-colouring \(\chi\) of
\(K_n\) such that
\[
\nu_H(\chi)\le\frac{1}{3a+b}n+O_{a,b}(1).
\]
\end{enumerate}
\end{lemma}

\begin{proof}
If \(b\ge3a\), partition \(V(K_n)\) into three parts of size
\(n/3+O(1)\). Colour the edges inside \(V_i\) and between \(V_i,V_{i+1}\)
with colour \(i\), with indices modulo \(3\). Each colour is incident with
at most \(2n/3+O(1)\) vertices, and hence supports at most
\(2n/[3(a+b)]+O_{a,b}(1)\) disjoint copies of \(H\).

Now assume \(b\le3a\). Take
\(|V_2|=|V_3|=\lfloor an/(3a+b)\rfloor\), and let \(V_1\) contain the
remaining vertices. Colour all edges inside \(V_3\) and between \(V_3\) and
\(V_1\cup V_2\) with colour \(3\); colour all edges inside \(V_2\) and
between \(V_2,V_1\) with colour \(2\); and colour all edges inside \(V_1\)
with colour \(1\).

In the colour-\(i\) graph, for \(i\in\{2,3\}\), the vertices outside
\(V_i\) form an independent set. Hence the vertices of any colour-\(i\)
copy of \(H\) lying in \(V_i\) form a vertex cover of \(H\), and therefore
there are at least \(a\) of them. Such a tiling has at most
\(|V_i|/a\le n/(3a+b)+O(1)\) copies. Colour \(1\) occurs only inside
\(V_1\), so a colour-\(1\) tiling has at most
\(|V_1|/(a+b)\le n/(3a+b)+O(1)\) copies.
\end{proof}

\begin{figure}[htbp]
\centering
\begin{tikzpicture}[scale=0.75]

\definecolor{EdgeLR}{RGB}{235,140,135}
\definecolor{EdgeTR}{RGB}{140,165,230}
\definecolor{EdgeTL}{RGB}{140,205,165}

\begin{scope}[xshift=-3.36cm]
 \coordinate (T2) at (0,1.76);
 \coordinate (L2) at (-1.6,-0.8);
 \coordinate (R2) at ( 1.6,-0.8);

 \tikzset{
  partBig/.style={
   circle, draw=black, line width=1.2pt,
   minimum size=24mm, inner sep=0pt, font=\small
  },
  partSmall/.style={
   circle, draw=black, line width=1.2pt,
   minimum size=12.8mm, inner sep=0pt, font=\small
  }
 }

 \draw[line width=12.8pt, line cap=round, color=EdgeTL] (R2) -- (L2);

 \foreach \i in {0,...,29} {
  \pgfmathsetmacro{\t}{\i/30}
  \pgfmathsetmacro{\tt}{(\i+1)/30}
  \pgfmathsetmacro{\w}{35.2 - (35.2-9.6)*\i/29}
  \draw[line cap=round, color=EdgeTL, line width=\w pt]
   ($(T2)!\t!(R2)$) -- ($(T2)!\tt!(R2)$);
 }

 \foreach \i in {0,...,29} {
  \pgfmathsetmacro{\t}{\i/30}
  \pgfmathsetmacro{\tt}{(\i+1)/30}
  \pgfmathsetmacro{\w}{35.2 - (35.2-9.6)*\i/29}
  \draw[line cap=round, color=EdgeTR, line width=\w pt]
   ($(T2)!\t!(L2)$) -- ($(T2)!\tt!(L2)$);
 }

 \node[partBig,  fill=EdgeLR] at (T2) {$V_1$};
 \node[partSmall, fill=EdgeTR] at (L2) {$V_2$};
 \node[partSmall, fill=EdgeTL] at (R2) {$V_3$};

 \node[font=\small] at (0,-1.96) {$b\le 3a$};
\end{scope}

\begin{scope}[xshift=3.36cm]
 \coordinate (T1) at (0,1.76);
 \coordinate (L1) at (-1.6,-0.8);
 \coordinate (R1) at ( 1.6,-0.8);

 \tikzset{
  part/.style={
   circle, draw=black, line width=1.2pt,
   minimum size=19.2mm, inner sep=0pt, font=\small
  }
 }

 \draw[line width=22.4pt, line cap=round, color=EdgeLR] (T1) -- (L1);
 \draw[line width=22.4pt, line cap=round, color=EdgeTR] (L1) -- (R1);
 \draw[line width=22.4pt, line cap=round, color=EdgeTL] (R1) -- (T1);

 \node[part, fill=EdgeLR] at (T1) {$V_1$};
 \node[part, fill=EdgeTR] at (L1) {$V_2$};
 \node[part, fill=EdgeTL] at (R1) {$V_3$};

 \node[font=\small] at (0,-1.96) {$b\ge 3a$};
\end{scope}

\end{tikzpicture}
\caption{\footnotesize Extremal \(3\)-colourings for Lemma~\ref{lem:upper3colour}.}
\label{fig:3strip}
\end{figure}
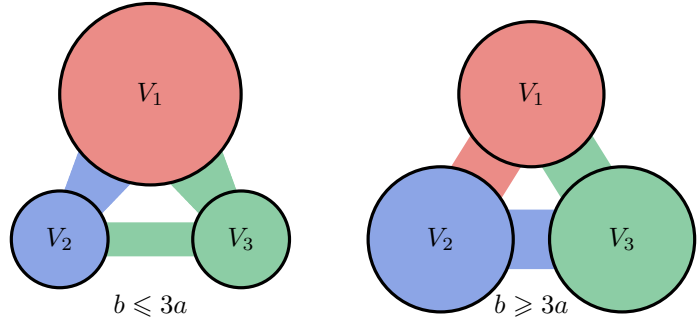

\begin{lemma}\label{lem:Kab-bipartite-packing}
Let \(b\ge a\ge1\) and \(N_1\ge N_2\). The complete bipartite graph
\(K_{N_1,N_2}\) contains a \(K_{a,b}\)-tiling of size at least
\[
\min\left\{
\frac{N_1+N_2}{a+b},
\frac{N_2}{a}
\right\}
-O_{a,b}(1).
\]
\end{lemma}

\begin{proof}
If \(a=b\), the assertion follows by greedily taking
\(\lfloor N_2/a\rfloor\) copies. Assume \(b>a\). If
\(aN_1>bN_2\), orient every copy with \(b\) vertices in the first class
and \(a\) in the second; this gives \(\lfloor N_2/a\rfloor\) copies.

Suppose instead that \(aN_1\le bN_2\). Set
\[
x:=\frac{bN_1-aN_2}{b^2-a^2},
\qquad
y:=\frac{bN_2-aN_1}{b^2-a^2}.
\]
Both quantities are non-negative. Taking \(\lfloor x\rfloor\) copies with
\(b\) vertices in the first class and \(\lfloor y\rfloor\) copies with
the opposite orientation respects both vertex capacities and gives
\[
x+y-O(1)=\frac{N_1+N_2}{a+b}-O(1)
\]
copies.
\end{proof}

\begin{lemma}\label{lem:lower3colour}
Let \(b\ge a\ge1\), and let \(\chi\) be a \(3\)-strip \(3\)-colouring of
\(K_n\). Then
\[
\nu_{K_{a,b}}(\chi)\ge
\begin{cases}
\displaystyle
\frac{1}{3a+b}n-O_{a,b}(1),
& b\le3a,\\[2ex]
\displaystyle
\frac{2}{3(a+b)}n-O_{a,b}(1),
& b\ge3a.
\end{cases}
\]
\end{lemma}

\begin{proof}
Pad the strip partition with empty parts if necessary, and write
\(V(K_n)=V_1\cup V_2\cup V_3\), where
\(n_1:=|V_1|\ge n_2:=|V_2|\ge n_3:=|V_3|\). Each graph induced by one
part, and each bipartite graph between two parts, is monochromatic.

Suppose first that \(b\ge3a\). If \(n_1\ge2n/3\), the monochromatic clique
on \(V_1\) contains \(2n/[3(a+b)]-O(1)\) disjoint copies of
\(K_{a,b}\). Otherwise \(n_1<2n/3\). Since
\(n=n_1+n_2+n_3\le n_1+2n_2\), we have \(n_2>n/6\). By
Lemma~\ref{lem:Kab-bipartite-packing}, the monochromatic graph between
\(V_1,V_2\) contains at least
\[
\min\left\{
\frac{n_1+n_2}{a+b},
\frac{n_2}{a}
\right\}
-O(1)
\]
copies. Here \(n_1+n_2\ge2n/3\), and
\(n_2/a>n/(6a)\ge2n/[3(a+b)]\), using \(b\ge3a\). This proves the first
case.

Now suppose that \(b\le3a\). If
\(n_1\ge(a+b)n/(3a+b)\), the clique on \(V_1\) gives
\(n/(3a+b)-O(1)\) copies. Otherwise the inequality
\(n\le n_1+2n_2\) implies \(n_2\ge an/(3a+b)\). Again applying
Lemma~\ref{lem:Kab-bipartite-packing}, the graph between \(V_1,V_2\)
contains at least the minimum of
\((n_1+n_2)/(a+b)\) and \(n_2/a\), up to \(O(1)\).
The second term is at least \(n/(3a+b)\), while
\(n_1+n_2\ge2n/3\) and \(b\le3a\) imply
\[
\frac{n_1+n_2}{a+b}
\ge
\frac{2n}{3(a+b)}
\ge
\frac{n}{3a+b}.
\]
This proves the second case.
\end{proof}

\begin{proof}[Proof of Theorem~\ref{thm:biparr=3}]
Lemma~\ref{lem:upper3colour} gives the matching upper bounds for
$\beta_{3,H}$. Since $H\subseteq K_{a,b}$, Lemma~\ref{lem:lower3colour}
gives the reverse bounds for every strip template. Hence $\beta_{3,H}$ has
the displayed value, and the result follows from
Theorem~\ref{thm:finite-optimization}.
\end{proof}

\subsubsection{A non-Hall bipartite example}
\label{subsec:D23-main}

The Hall-type condition in Theorem~\ref{thm:biparr=3} leads to a uniform
formula, but it is not necessary for an explicit evaluation. Let \(D_{2,3}\)
be the double star whose two centres are adjacent, with two leaves attached
to one centre and three leaves attached to the other. Its bipartition classes
have sizes \(3\) and \(4\), but its maximum matching has size \(2\).
Consequently, \(D_{2,3}\) does not satisfy the Hall-type condition.

Nevertheless, a direct analysis of its three-part strip templates determines
$\beta_{3,D_{2,3}}$ exactly.

\begin{theorem}\label{thm:D23}
\[
Rt_3(D_{2,3};K_n)=\left(\frac{7}{78}+o(1)\right)n.
\]
\end{theorem}

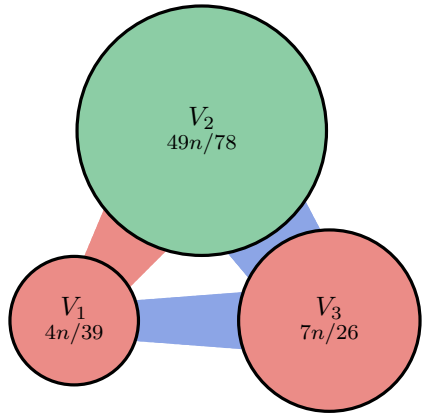
\begin{figure}[htbp]
\centering
\begin{tikzpicture}[scale=0.75]

\definecolor{EdgeOne}{RGB}{235,140,135}
\definecolor{EdgeTwo}{RGB}{140,165,230}
\definecolor{EdgeThree}{RGB}{140,205,165}

\begin{scope}[xshift=0cm]
 \coordinate (T) at (0,2.35);
 \coordinate (L) at (-2.25,-1.0);
 \coordinate (R) at ( 2.25,-1.0);

 \tikzset{
  partBig/.style={
   circle, draw=black, line width=1.2pt,
   minimum size=33mm, inner sep=0pt, font=\small, align=center
  },
  partMid/.style={
   circle, draw=black, line width=1.2pt,
   minimum size=24mm, inner sep=0pt, font=\small, align=center
  },
  partSmall/.style={
   circle, draw=black, line width=1.2pt,
   minimum size=17mm, inner sep=0pt, font=\small, align=center
  }
 }

 \foreach \i in {0,...,29} {
  \pgfmathsetmacro{\t}{\i/30}
  \pgfmathsetmacro{\tt}{(\i+1)/30}
  \pgfmathsetmacro{\w}{44 - (44-12)*\i/29}
  \draw[line cap=round, color=EdgeOne, line width=\w pt]
   ($(T)!\t!(L)$) -- ($(T)!\tt!(L)$);
 }

 \foreach \i in {0,...,29} {
  \pgfmathsetmacro{\t}{\i/30}
  \pgfmathsetmacro{\tt}{(\i+1)/30}
  \pgfmathsetmacro{\w}{44 - (44-26)*\i/29}
  \draw[line cap=round, color=EdgeTwo, line width=\w pt]
   ($(T)!\t!(R)$) -- ($(T)!\tt!(R)$);
 }

 \foreach \i in {0,...,29} {
  \pgfmathsetmacro{\t}{\i/30}
  \pgfmathsetmacro{\tt}{(\i+1)/30}
  \pgfmathsetmacro{\w}{26 - (26-12)*\i/29}
  \draw[line cap=round, color=EdgeTwo, line width=\w pt]
   ($(R)!\t!(L)$) -- ($(R)!\tt!(L)$);
 }

 \node[partSmall, fill=EdgeOne] at (L)
  {$V_1$\\[-1mm]\scriptsize \(4n/39\)};

 \node[partBig, fill=EdgeThree] at (T)
  {$V_2$\\[-1mm]\scriptsize \(49n/78\)};

 \node[partMid, fill=EdgeOne] at (R)
  {$V_3$\\[-1mm]\scriptsize \(7n/26\)};

\end{scope}

\end{tikzpicture}
\caption{\footnotesize The \(D_{2,3}\) construction.}
\label{fig:D23upper}
\end{figure}

For comparison, the complete bipartite graph \(K_{3,4}\), which has the same
bipartition sizes, has three-colour constant \(1/13\) by
Theorem~\ref{thm:biparr=3}. Thus the tiling constant is not determined by the
bipartition sizes alone once the Hall-type condition is dropped. The
construction and strip calculation in Appendix~\ref{app:D23} show that
$\beta_{3,D_{2,3}}=7/78$; Theorem~\ref{thm:finite-optimization} then
gives the displayed asymptotic formula.

\subsubsection{Complete bipartite tilings with four and five colours}

The four-colour case is the first in which the optimal value has more than
two regimes. It also illustrates how the finite strip-template optimization
can be solved without enumerating all templates.

\begin{theorem}\label{thm:Kab-four}
Let \(b\ge a\ge1\). Then
\[
Rt_4(K_{a,b};K_n)=
\begin{cases}
\displaystyle
\left(\frac{1}{4a+b}+o(1)\right)n,
& b\le3a,\\[2ex]
\displaystyle
\left(\frac{2}{5a+3b}+o(1)\right)n,
& 3a\le b\le5a,\\[2ex]
\displaystyle
\left(\frac{3}{5(a+b)}+o(1)\right)n,
& b\ge5a.
\end{cases}
\]
\end{theorem}

The upper bounds are given by the following three strip templates:
\[
\Lambda_1=
\begin{pmatrix}
1&2&3&4\\
2&2&3&4\\
3&3&3&4\\
4&4&4&4
\end{pmatrix},
\qquad
\Lambda_2=
\begin{pmatrix}
1&1&3&4\\
1&2&2&4\\
3&2&3&4\\
4&4&4&4
\end{pmatrix},
\qquad
\Lambda_3=
\begin{pmatrix}
1&1&2&3\\
1&4&4&4\\
2&4&4&4\\
3&4&4&4
\end{pmatrix}.
\]
Here the corresponding normalized part-size vectors are
\[
\frac{(a+b,a,a,a)}{4a+b},
\qquad
\frac{(a+b,a+b,a+b,2a)}{5a+3b},
\qquad
\frac{(2,1,1,1)}5,
\]
respectively.

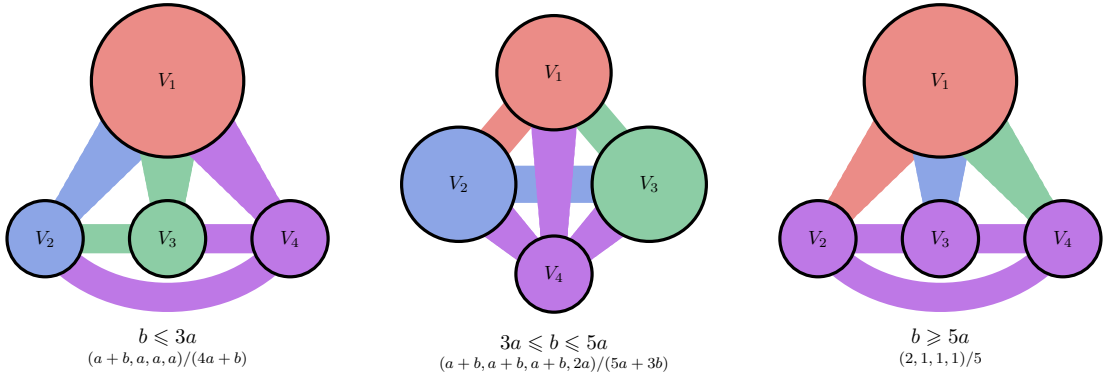
\begin{figure}[htbp]
\centering
\begin{tikzpicture}[
 scale=0.72,
 transform shape,
 line cap=round,
 line join=round
]

\definecolor{KabRed}{RGB}{235,140,135}
\definecolor{KabBlue}{RGB}{140,165,230}
\definecolor{KabGreen}{RGB}{140,205,165}
\definecolor{KabPurple}{RGB}{190,120,230}

\tikzset{
 KabBig/.style={
  circle, draw=black, line width=1.2pt,
  minimum size=28mm, inner sep=0pt,
  font=\small, align=center
 },
 KabMedium/.style={
  circle, draw=black, line width=1.2pt,
  minimum size=21mm, inner sep=0pt,
  font=\small, align=center
 },
 KabSmall/.style={
  circle, draw=black, line width=1.2pt,
  minimum size=14mm, inner sep=0pt,
  font=\small, align=center
 }
}

\def\KabTaper#1#2#3#4#5{%
 \foreach \i in {0,...,19} {
  \pgfmathsetmacro{\KabStart}{\i/20}
  \pgfmathsetmacro{\KabEnd}{(\i+1)/20}
  \pgfmathsetmacro{\KabWidth}{#4-(#4-#5)*\i/19}
  \draw[color=#3,line width=\KabWidth pt]
   ($(#1)!\KabStart!(#2)$) --
   ($(#1)!\KabEnd!(#2)$);
 }
}

\begin{scope}[xshift=-7.1cm]
 \coordinate (A1) at (0,2.15);
 \coordinate (A2) at (-2.25,-0.75);
 \coordinate (A3) at (0,-0.75);
 \coordinate (A4) at (2.25,-0.75);

 \draw[color=KabGreen,line width=11pt] (A2)--(A3);
 \draw[color=KabPurple,line width=11pt] (A3)--(A4);
 \draw[color=KabPurple,line width=11pt]
  (A2) to[bend right=55] (A4);

 \KabTaper{A1}{A2}{KabBlue}{28}{8}
 \KabTaper{A1}{A3}{KabGreen}{30}{8}
 \KabTaper{A1}{A4}{KabPurple}{28}{8}

 \node[KabBig,fill=KabRed] at (A1) {$V_1$};
 \node[KabSmall,fill=KabBlue] at (A2) {$V_2$};
 \node[KabSmall,fill=KabGreen] at (A3) {$V_3$};
 \node[KabSmall,fill=KabPurple] at (A4) {$V_4$};

 \node[align=center] at (0,-2.75)
 {$b\le3a$\\[-1mm]
  \scriptsize $(a+b,a,a,a)/(4a+b)$};
\end{scope}

\begin{scope}[xshift=0cm,yshift=0.25cm]
 \coordinate (B1) at (0,2.05);
 \coordinate (B2) at (-1.75,0);
 \coordinate (B3) at (1.75,0);
 \coordinate (B4) at (0,-1.65);

 \draw[color=KabRed,line width=14pt] (B1)--(B2);
 \draw[color=KabGreen,line width=14pt] (B1)--(B3);
 \draw[color=KabBlue,line width=14pt] (B2)--(B3);

 \KabTaper{B1}{B4}{KabPurple}{20}{8}
 \KabTaper{B2}{B4}{KabPurple}{20}{8}
 \KabTaper{B3}{B4}{KabPurple}{20}{8}

 \node[KabMedium,fill=KabRed] at (B1) {$V_1$};
 \node[KabMedium,fill=KabBlue] at (B2) {$V_2$};
 \node[KabMedium,fill=KabGreen] at (B3) {$V_3$};
 \node[KabSmall,fill=KabPurple] at (B4) {$V_4$};

 \node[align=center] at (0,-3.1)
 {$3a\le b\le5a$\\[-1mm]
  \scriptsize $(a+b,a+b,a+b,2a)/(5a+3b)$};
\end{scope}

\begin{scope}[xshift=7.1cm]
 \coordinate (C1) at (0,2.15);
 \coordinate (C2) at (-2.25,-0.75);
 \coordinate (C3) at (0,-0.75);
 \coordinate (C4) at (2.25,-0.75);

 \draw[color=KabPurple,line width=11pt] (C2)--(C3);
 \draw[color=KabPurple,line width=11pt] (C3)--(C4);
 \draw[color=KabPurple,line width=11pt]
  (C2) to[bend right=55] (C4);

 \KabTaper{C1}{C2}{KabRed}{28}{8}
 \KabTaper{C1}{C3}{KabBlue}{30}{8}
 \KabTaper{C1}{C4}{KabGreen}{28}{8}

 \node[KabBig,fill=KabRed] at (C1) {$V_1$};
 \node[KabSmall,fill=KabPurple] at (C2) {$V_2$};
 \node[KabSmall,fill=KabPurple] at (C3) {$V_3$};
 \node[KabSmall,fill=KabPurple] at (C4) {$V_4$};

 \node[align=center] at (0,-2.75)
 {$b\ge5a$\\[-1mm]
  \scriptsize $(2,1,1,1)/5$};
\end{scope}

\end{tikzpicture}
\caption{\footnotesize Extremal four-colour strip templates for
\(K_{a,b}\)-tilings.}
\label{fig:Kab4-extremal}
\end{figure}

For \(\Lambda_1\), colour \(1\) occurs only inside the first part, while
every edge of colour \(j\in\{2,3,4\}\) meets the \(j\)-th part. For
\(\Lambda_2\), each of colours \(1,2,3\) is supported on two of the three
large parts, while every colour-\(4\) edge meets the fourth part. For
\(\Lambda_3\), every colour is supported on parts of total normalized size
at most \(3/5\). These observations give the three upper bounds.

We now prove the matching lower bounds simultaneously.

\begin{lemma}\label{lem:Kab-four-forcing}
Let \(\lambda>0\), and put
$A:=(a+b)\lambda,$
$B:=a\lambda.$
Suppose that
$A+3B\le1,$
$1-\frac{3A}{2}\ge B,$
and 
$A\le\frac35.$
Then every strip colouring with at most four parts and at most four colours
on its cross-pairs contains a monochromatic \(K_{a,b}\)-tiling of size
$\lambda n-O_{a,b}(1).$
The colours inside the individual parts may be arbitrary.
\end{lemma}

\begin{proof}
Pad the strip partition with empty parts if necessary, and let its normalized
part sizes be
$x_1\ge x_2\ge x_3\ge x_4,$
$x_1+x_2+x_3+x_4=1.$

Suppose first that \(x_1\ge(b/a)x_2\). If \(x_1\ge A\), the monochromatic
clique on the first part gives the required tiling. If \(x_2\ge B\), the
complete bipartite graph between the first two parts gives it. Otherwise,
$1\le x_1+3x_2<A+3B\le1,$
a contradiction.

We may therefore assume that \(x_1<(b/a)x_2\). By
Lemma~\ref{lem:Kab-bipartite-packing}, the first two parts give the required
tiling unless
$x_1+x_2<A.$
Since \(x_3\le x_2\le(x_1+x_2)/2\), we have
\[
x_4
=
1-x_1-x_2-x_3
>
1-\frac{3A}{2}
\ge B.
\]
If \(x_1\ge(b/a)x_4\), the first and fourth parts give the required
tiling. Hence we may assume
$x_1<\frac ba x_4.$
Thus every pair of parts has size ratio smaller than \(b/a\).

Moreover,
$2x_1+1
=
3x_1+x_2+x_3+x_4
\le
3(x_1+x_2)
<
3A.$
It follows that every three parts have total normalized size greater than
\[
1-\frac{3A-1}{2}
=
\frac32(1-A)
\ge A.
\]

Among the six cross-pairs, some colour occurs at least twice. If two
same-coloured pairs are adjacent in the auxiliary \(K_4\), the common part
and the union of the other two parts form a monochromatic complete bipartite
graph. Both sides have size at least \(B\), and their total size is at least
\(A\), so Lemma~\ref{lem:Kab-bipartite-packing} gives the required tiling.

If the two pairs are disjoint, each pair can be tiled using all its vertices
up to \(O_{a,b}(1)\), because its part-size ratio is smaller than \(b/a\).
The two tilings have the same colour and together contain
$\frac{n}{a+b}-O_{a,b}(1)
\ge
\lambda n-O_{a,b}(1)$
copies.
\end{proof}

\begin{proof}[Proof of Theorem~\ref{thm:Kab-four}]
The three templates above give the upper bounds for
$\beta_{4,K_{a,b}}$. Apply Lemma~\ref{lem:Kab-four-forcing} with
\[
\lambda=\frac1{4a+b},\qquad
\lambda=\frac2{5a+3b},\qquad
\lambda=\frac3{5(a+b)}
\]
in the three respective ranges. Direct substitution shows that its
hypotheses hold, so
the lemma gives the matching lower bounds for the strip optimum. Hence
$\beta_{4,K_{a,b}}$ has the displayed value, and
Theorem~\ref{thm:finite-optimization} gives the result.
\end{proof}

\begin{theorem}\label{thm:Kab-five}
Let \(b\ge a\ge1\). Then
\[
Rt_5(K_{a,b};K_n)=
\begin{cases}
\displaystyle
\left(\frac{1}{5a+b}+o(1)\right)n,
& b\le3a,\\[2ex]
\displaystyle
\left(\frac{2}{7a+3b}+o(1)\right)n,
& 3a\le b\le5a,\\[2ex]
\displaystyle
\left(\frac{3}{8a+5b}+o(1)\right)n,
& 5a\le b\le\frac{13}{2}a,\\[2ex]
\displaystyle
\left(\frac{5}{9(a+b)}+o(1)\right)n,
& b\ge\frac{13}{2}a.
\end{cases}
\]
\end{theorem}

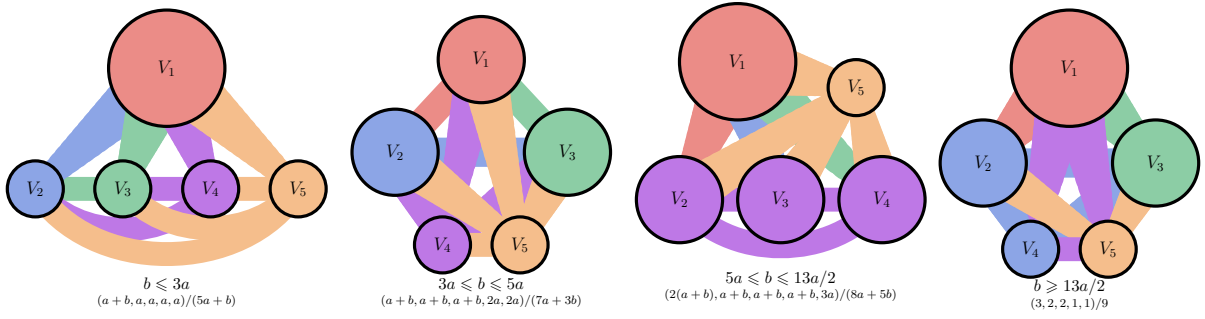
\begin{figure}[htbp]
\centering
\begin{tikzpicture}[
 scale=0.57,
 transform shape,
 line cap=round,
 line join=round
]

\definecolor{KabRed}{RGB}{235,140,135}
\definecolor{KabBlue}{RGB}{140,165,230}
\definecolor{KabGreen}{RGB}{140,205,165}
\definecolor{KabPurple}{RGB}{190,120,230}
\definecolor{KabOrange}{RGB}{245,190,140}

\tikzset{
 KabFiveLarge/.style={
  circle, draw=black, line width=1.2pt,
  minimum size=27mm, inner sep=0pt,
  font=\large, align=center
 },
 KabFiveMedium/.style={
  circle, draw=black, line width=1.2pt,
  minimum size=20mm, inner sep=0pt,
  font=\large, align=center
 },
 KabFiveSmall/.style={
  circle, draw=black, line width=1.2pt,
  minimum size=13mm, inner sep=0pt,
  font=\large, align=center
 }
}

\def\KabFiveTaper#1#2#3#4#5{%
 \foreach \i in {0,...,19} {
  \pgfmathsetmacro{\KabFiveStart}{\i/20}
  \pgfmathsetmacro{\KabFiveEnd}{(\i+1)/20}
  \pgfmathsetmacro{\KabFiveWidth}{#4-(#4-#5)*\i/19}
  \draw[color=#3,line width=\KabFiveWidth pt]
   ($(#1)!\KabFiveStart!(#2)$) --
   ($(#1)!\KabFiveEnd!(#2)$);
 }
}

\begin{scope}[xshift=-10.6cm]
 \coordinate (A1) at (0,2.2);
 \coordinate (A2) at (-3.05,-0.6);
 \coordinate (A3) at (-1.02,-0.6);
 \coordinate (A4) at (1.02,-0.6);
 \coordinate (A5) at (3.05,-0.6);

 \draw[color=KabGreen,line width=9pt] (A2)--(A3);
 \draw[color=KabPurple,line width=9pt] (A3)--(A4);
 \draw[color=KabPurple,line width=9pt]
  (A2) to[bend right=52] (A4);

 \KabFiveTaper{A1}{A2}{KabBlue}{23}{7}
 \KabFiveTaper{A1}{A3}{KabGreen}{25}{7}
 \KabFiveTaper{A1}{A4}{KabPurple}{23}{7}
 \KabFiveTaper{A1}{A5}{KabOrange}{23}{7}

 \draw[color=KabOrange,line width=9pt] (A4)--(A5);
 \draw[color=KabOrange,line width=9pt]
  (A3) to[bend right=50] (A5);
 \draw[color=KabOrange,line width=9pt]
  (A2) to[bend right=58] (A5);

 \node[KabFiveLarge,fill=KabRed] at (A1) {$V_1$};
 \node[KabFiveSmall,fill=KabBlue] at (A2) {$V_2$};
 \node[KabFiveSmall,fill=KabGreen] at (A3) {$V_3$};
 \node[KabFiveSmall,fill=KabPurple] at (A4) {$V_4$};
 \node[KabFiveSmall,fill=KabOrange] at (A5) {$V_5$};

 \node[align=center] at (0,-3)
 {$b\le3a$\\[-1mm]
  \scriptsize $(a+b,a,a,a,a)/(5a+b)$};
\end{scope}

\begin{scope}[xshift=-3.3cm,yshift=0.25cm]
 \coordinate (B1) at (0,2.15);
 \coordinate (B2) at (-2.0,0);
 \coordinate (B3) at (2.0,0);
 \coordinate (B4) at (-0.9,-2.15);
 \coordinate (B5) at (0.9,-2.15);

 \draw[color=KabRed,line width=11pt] (B1)--(B2);
 \draw[color=KabGreen,line width=11pt] (B1)--(B3);
 \draw[color=KabBlue,line width=11pt] (B2)--(B3);

 \KabFiveTaper{B1}{B4}{KabPurple}{18}{7}
 \KabFiveTaper{B2}{B4}{KabPurple}{18}{7}
 \KabFiveTaper{B3}{B4}{KabPurple}{18}{7}

 \KabFiveTaper{B1}{B5}{KabOrange}{18}{7}
 \KabFiveTaper{B2}{B5}{KabOrange}{18}{7}
 \KabFiveTaper{B3}{B5}{KabOrange}{18}{7}

 \draw[color=KabOrange,line width=9pt] (B4)--(B5);

 \node[KabFiveMedium,fill=KabRed] at (B1) {$V_1$};
 \node[KabFiveMedium,fill=KabBlue] at (B2) {$V_2$};
 \node[KabFiveMedium,fill=KabGreen] at (B3) {$V_3$};
 \node[KabFiveSmall,fill=KabPurple] at (B4) {$V_4$};
 \node[KabFiveSmall,fill=KabOrange] at (B5) {$V_5$};

 \node[align=center] at (0,-3.25)
 {$3a\le b\le5a$\\[-1mm]
  \scriptsize $(a+b,a+b,a+b,2a,2a)/(7a+3b)$};
\end{scope}

\begin{scope}[xshift=3.65cm]
 \coordinate (C1) at (-1,2.35);
 \coordinate (C2) at (-2.35,-0.85);
 \coordinate (C3) at (0,-0.85);
 \coordinate (C4) at (2.35,-0.85);
 \coordinate (C5) at (1.75,1.75);

 \draw[color=KabPurple,line width=9pt] (C2)--(C3);
 \draw[color=KabPurple,line width=9pt] (C3)--(C4);
 \draw[color=KabPurple,line width=9pt]
  (C2) to[bend right=52] (C4);

 \KabFiveTaper{C1}{C2}{KabRed}{23}{7}
 \KabFiveTaper{C1}{C3}{KabBlue}{25}{7}
 \KabFiveTaper{C1}{C4}{KabGreen}{23}{7}

 \KabFiveTaper{C1}{C5}{KabOrange}{24}{7}
 \KabFiveTaper{C2}{C5}{KabOrange}{18}{7}
 \KabFiveTaper{C3}{C5}{KabOrange}{18}{7}
 \KabFiveTaper{C4}{C5}{KabOrange}{18}{7}

 \node[KabFiveLarge,fill=KabRed] at (C1) {$V_1$};
 \node[KabFiveMedium,fill=KabPurple] at (C2) {$V_2$};
 \node[KabFiveMedium,fill=KabPurple] at (C3) {$V_3$};
 \node[KabFiveMedium,fill=KabPurple] at (C4) {$V_4$};
 \node[KabFiveSmall,fill=KabOrange] at (C5) {$V_5$};

 \node[align=center] at (0,-2.85)
 {$5a\le b\le13a/2$\\[-1mm]
  \scriptsize
  $(2(a+b),a+b,a+b,a+b,3a)/(8a+5b)$};
\end{scope}

\begin{scope}[xshift=10.35cm]
 \coordinate (D1) at (0,2.2);
 \coordinate (D2) at (-2,0);
 \coordinate (D3) at (2,0);
 \coordinate (D4) at (-0.9,-2);
 \coordinate (D5) at (0.9,-2);

 \KabFiveTaper{D1}{D2}{KabRed}{26}{8}
 \KabFiveTaper{D1}{D3}{KabGreen}{26}{8}

 \draw[color=KabBlue,line width=11pt] (D2)--(D3);

 \KabFiveTaper{D1}{D4}{KabPurple}{22}{7}
 \KabFiveTaper{D2}{D4}{KabBlue}{17}{7}
 \KabFiveTaper{D3}{D4}{KabBlue}{17}{7}

 \KabFiveTaper{D1}{D5}{KabPurple}{22}{7}
 \KabFiveTaper{D2}{D5}{KabOrange}{17}{7}
 \KabFiveTaper{D3}{D5}{KabOrange}{17}{7}

 \draw[color=KabPurple,line width=9pt] (D4)--(D5);

 \node[KabFiveLarge,fill=KabRed] at (D1) {$V_1$};
 \node[KabFiveMedium,fill=KabBlue] at (D2) {$V_2$};
 \node[KabFiveMedium,fill=KabGreen] at (D3) {$V_3$};
 \node[KabFiveSmall,fill=KabBlue] at (D4) {$V_4$};
 \node[KabFiveSmall,fill=KabOrange] at (D5) {$V_5$};

 \node[align=center] at (0,-3.05)
 {$b\ge13a/2$\\[-1mm]
  \scriptsize $(3,2,2,1,1)/9$};
\end{scope}

\end{tikzpicture}
\caption{\footnotesize Extremal five-colour strip templates for
\(K_{a,b}\)-tilings.}
\label{fig:Kab5-extremal}
\end{figure}

The five-colour optimization has one additional regime and uses
Theorem~\ref{thm:Kab-four} recursively inside a four-part subtemplate. The
calculation in Appendix~\ref{subsec:5colour} identifies the displayed value
of $\beta_{5,K_{a,b}}$; Theorem~\ref{thm:finite-optimization} then gives the displayed
asymptotic formula.

These formulas suggest that, although the finite optimization applies to
every fixed number of colours, extracting a uniform closed formula for
\(K_{a,b}\) requires a more conceptual understanding of the extremal strip
templates.

\subsubsection{Hall-type balanced bipartite tilings}

\begin{theorem}\label{thm:Kaa}
Let \(r\ge1\), and let \(H\) be a connected bipartite graph whose two
bipartition classes both have size \(a\), and which contains a matching of
size \(a\). Then
\[
Rt_r(H;K_n)=\left(\frac{1}{a(r+1)}+o(1)\right)n.
\]
\end{theorem}

In particular, this applies to the \(d\)-dimensional hypercube.

\begin{corollary}\label{cor:Qd}
For every \(r\ge1\) and \(d\ge2\),
\[
Rt_r(Q_d;K_n)=\left(\frac{1}{2^{d-1}(r+1)}+o(1)\right)n.
\]
\end{corollary}

\begin{proof}
The two bipartition classes of \(Q_d\) both have size \(2^{d-1}\), and
\(Q_d\) has a perfect matching.
\end{proof}

\begin{lemma}\label{lem:upperKaa}
Let \(r\ge1\) and \(H\) be as in Theorem~\ref{thm:Kaa}. There is
an \(r\)-colouring \(\chi\) of \(K_n\) such that
\[
\nu_H(\chi)\le\frac{1}{a(r+1)}n+O_a(1).
\]
\end{lemma}

\begin{proof}
Partition \(V(K_n)\) into \(V_1,\ldots,V_r\), with
\(|V_i|=n/(r+1)+O(1)\) for \(i<r\) and
\(|V_r|=2n/(r+1)+O(1)\). Colour the edges inside \(V_i\) with colour
\(i\), and colour all edges between \(V_i,V_j\), where \(i<j\), with
colour \(i\).

For \(i<r\), the vertices outside \(V_i\) form an independent set in
colour \(i\). Hence every colour-\(i\) copy of \(H\) uses a vertex cover
of \(H\) inside \(V_i\). Since \(H\) has a matching of size \(a\) and a
bipartition class of size \(a\), its vertex-cover number is exactly \(a\).
Thus a colour-\(i\) tiling has at most \(|V_i|/a\) copies. Colour \(r\)
occurs only inside \(V_r\), so a colour-\(r\) tiling has at most
\(|V_r|/(2a)\) copies. Both bounds are
\(n/[a(r+1)]+O_a(1)\).
\end{proof}

\begin{lemma}\label{lem:lowerKaa}
Let \(r\ge1\) and \(a\ge1\). Every \(r\)-strip colouring of
\(K_n\) contains at least
$\left\lfloor\frac{n}{a(r+1)}\right\rfloor$
disjoint monochromatic copies of \(K_{a,a}\).
\end{lemma}

\begin{proof}
The case \(r=1\) is immediate, so assume \(r\ge2\).

Let \(U_1,\ldots,U_q\), where \(q\le r\), be the strip parts, and set
\(x_i:=(r+1)|U_i|/n\). If some \(x_i\ge2\), the monochromatic clique on
\(U_i\) contains at least \(\lfloor |U_i|/(2a)\rfloor\ge
\lfloor n/[a(r+1)]\rfloor\) disjoint copies of \(K_{a,a}\).

We may therefore assume \(x_i<2\) for every \(i\). Since
\(\sum_i x_i=r+1\), the quantity
$\sum_{i=1}^q\bigl(x_i-\lfloor x_i\rfloor\bigr)
=
r+1-\sum_{i=1}^q\lfloor x_i\rfloor$
is an integer smaller than \(q\le r\). Hence
\(\sum_i\lfloor x_i\rfloor\ge2\). There are therefore two parts,
say \(U_1,U_2\), each of size at least \(n/(r+1)\). All edges between
them have one colour, so greedily taking \(a\) vertices from each part
produces at least
$\min\left\{
\left\lfloor\frac{|U_1|}{a}\right\rfloor,
\left\lfloor\frac{|U_2|}{a}\right\rfloor
\right\}
\ge
\left\lfloor\frac{n}{a(r+1)}\right\rfloor$
disjoint monochromatic copies of \(K_{a,a}\).
\end{proof}

\begin{proof}[Proof of Theorem~\ref{thm:Kaa}]
The case $r=1$ is immediate, so assume $r\ge2$. Lemma~\ref{lem:upperKaa}
gives $\beta_{r,H}\le1/[a(r+1)]$. Since $H\subseteq K_{a,a}$,
Lemma~\ref{lem:lowerKaa} gives the reverse inequality for every strip
template. Thus $\beta_{r,H}=1/[a(r+1)]$, and the result follows from
Theorem~\ref{thm:finite-optimization}.
\end{proof}

\section{Concluding remarks and further directions}\label{sec:conclusion}

We introduced a strip-colouring framework which reduces arbitrary
edge-colourings to finite templates while preserving large monochromatic
matchings or graph tilings. The asymptotic reduction applies to pseudorandom
hosts. At the critical matching scale, a rigidity, absorption, and switching
argument gives the exact AFL bound for locally almost-complete hosts and
settles the stable Kneser problem in the stated linear regime. For graph
tilings, the finite template optimization determines the asymptotic
coefficient for every fixed number of colours and every fixed target graph.
We conclude with several directions suggested by these results.

\subsection{Quantitative and growing-uniformity questions}

Theorem~\ref{thm:robust-critical-AFL} is an eventual result: the proof uses
regularity, a hierarchy of parameters, and a vanishing-error stability
argument, and gives no useful estimate for the threshold at which the exact
formula begins. It
would be interesting to obtain quantitative bounds, or a finitary stability
theorem describing every colouring whose largest monochromatic matching is
close to the critical value.

Our arguments also keep $t$ fixed. Both the weak regularity input and the
critical rigidity argument become ineffective when the uniformity grows
with $n$. Is there a topology-free proof of AFL, exact or asymptotic, which
remains effective for a nontrivial range of $t=t(n)$?

\subsection{Stable Kneser hypergraphs}

Theorem~\ref{thm:stable} proves Meunier's predicted formula for every fixed
$t\ge2$ and $\gamma>0$ when
$
\gamma n\le q, tq\le n,$ and $ s\le cq,
$
for a small constant $c=c(t,\gamma)$.
The case $q=o(n)$ is not covered by our argument, since the number of
colours required in the corresponding matching formulation is then
unbounded, whereas Theorem~\ref{thm:robust-critical-AFL} is proved for each
fixed number of colours. It would also be interesting to understand the
regime $s=\Theta(q)$ without the smallness assumption on the implied
constant, where the local density of the complement may exceed the
robustness threshold supplied by Theorem~\ref{thm:robust-critical-AFL}.

\subsection{Hypergraph tilings beyond matchings}

The matching problem treated in the hypergraph part of the paper is the
special case of tiling by a single \(t\)-edge. The graph part of the paper shows that, when \(t=2\), the same
LP-duality philosophy can handle much more general tiling problems. It is
therefore natural to ask how far this can be extended for \(t\)-uniform
hypergraphs with \(t\ge3\).

\begin{problem}\label{prob:hypergraph-tiling-strip}
Let \(r,t\ge3\). For which $t$-uniform hypergraphs $\mathcal H$ does there exist a finite
strip-template optimization determining
$
Rt_r(\mathcal H;K_n^{(t)})
=(\beta_{r,\mathcal H}+o(1))n?
$
\end{problem}

We believe that our method should also handle the case where $\mathcal H$ is the balanced $t$-partite $t$-graph $K^{(t)}_{a,\ldots,a}$, since $K_n^{(t)}$ has sufficient edge density to allow the use of a Turán-type theorem to embed $\mathcal H$.
In the case $r\ge3$ and $\cH=K_p^{(t)}$ for any $p>t$, the elementary
greedy lower bound is asymptotically sharp up to an additive constant.
We provide the proof in the Appendix.

\begin{proposition}\label{prop:clique-tiling-upper}
Let $t\ge 3$, $r\ge 3$, and $p\ge t+1$. Then
$Rt_r\left(K_p^{(t)};K_n^{(t)}\right)
=
\frac{n}{rp}+O_{r,t,p}(1).$
\end{proposition}

There is a further distinction between complete hosts and sparse pseudorandom
hosts. For matchings, weak regularity is sufficient. For general
\(\mathcal H\)-tilings, however, if the host hypergraph is sparse, weak regularity does not provide counting
lemmas for non-linear hypergraphs, that is, hypergraphs containing two edges
with more than one common vertex. This obstruction already appears for simple
examples such as \(K_{2,2,2}^{(3)}\). It would be interesting to develop a structural reduction for monochromatic \(\mathcal H\)-tilings in
sparse \(t\)-graphs, for fixed \(t\ge3\), which retains a finite
optimization description of the asymptotic tiling size.

\subsection{Explicit bipartite tiling constants}

Theorem~\ref{thm:finite-optimization} reduces the evaluation of
\(\beta_{r,H}\) to a finite optimization for every fixed \(r\) and \(H\),
but extracting a transparent closed formula remains a separate extremal
problem. For complete bipartite graphs, the two-colour formula follows from
the theorem of Burr, Erd\H{o}s and Spencer, while
Theorems~\ref{thm:biparr=3},~\ref{thm:Kab-four} and~\ref{thm:Kab-five} determine the cases
\(r=3,4,5\).

Already these first cases display increasing complexity. As the number of
colours increases from \(3\) to \(5\), the number of parameter regimes in the
ratio \(b/a\) increases from two to three and then four. Each breakpoint
corresponds to a change in the optimal strip template. Although the finite
algorithm applies for every fixed \(r\), enumerating all templates for larger
\(r\) becomes unwieldy and does not reveal why particular templates are
extremal.

\begin{problem}\label{prob:Kab-explicit}
For fixed \(r\ge6\), determine \(a\,\beta_{r,K_{a,b}}\) as a function of
the ratio \(b/a\). In particular, how many breakpoints can occur as a
function of \(r\), and can the extremal strip templates be characterized
without exhaustive case analysis?
\end{problem}

The double-star result shows that explicit evaluation is not confined to
Hall-type graphs. The graph \(D_{2,3}\) has bipartition classes of sizes
\(3\) and \(4\), but its maximum matching has size \(2\), so it lies outside
the scope of Theorem~\ref{thm:biparr=3}. Nevertheless,
Theorem~\ref{thm:D23} gives
$\beta_{3,D_{2,3}}=\frac{7}{78}.$
Since \(K_{3,4}\) has the same bipartition sizes but a different tiling
constant, the value of \(\beta_{3,H}\) for non-Hall bipartite graphs depends
on finer structural properties of \(H\).

\begin{problem}\label{prob:nonHall-bipartite}
Determine \(\beta_{3,H}\) for connected bipartite graphs \(H\) which do not
contain a matching saturating a smallest bipartition class. Which structural
parameters of \(H\), beyond its bipartition sizes and independence number,
govern the answer?
\end{problem}

\section*{Acknowledgement}
We thank Rob Morris for helpful discussions. This project was carried out during the IBS ECOPRO 2025 Autumn Student Research Experience Workshop. We acknowledge the use of AI for the optimizations done in Appendices~\ref{subsec:5colour} and~\ref{app:D23}.
HL, LW and ZY are supported by the Institute for Basic Science (IBS-R029-C4), and LW is also supported by National Key R\&D Program of China under grant number 2024YFA1013900, NSFC under grant number 12471327 and by China Scholarship Council.


\bibliographystyle{abbrv}
\bibliography{multicolourtiling}

\appendix

\section{Complete bipartite tilings in five colours}
\label{subsec:5colour}

We prove Theorem~\ref{thm:Kab-five}. We repeatedly use
Lemma~\ref{lem:Kab-bipartite-packing}, as well as the fact that tilings
supported on disjoint same-coloured pairs may be combined.

\subsection{Extremal templates}

The four upper-bound templates are
\[
\Lambda_1^{(5)}=
\begin{pmatrix}
1&2&3&4&5\\
2&2&3&4&5\\
3&3&3&4&5\\
4&4&4&4&5\\
5&5&5&5&5
\end{pmatrix},
\qquad
\Lambda_2^{(5)}=
\begin{pmatrix}
1&1&3&4&5\\
1&2&2&4&5\\
3&2&3&4&5\\
4&4&4&4&5\\
5&5&5&5&5
\end{pmatrix},
\]
\[
\Lambda_3^{(5)}=
\begin{pmatrix}
1&1&2&3&5\\
1&4&4&4&5\\
2&4&4&4&5\\
3&4&4&4&5\\
5&5&5&5&5
\end{pmatrix},
\qquad
\Lambda_4^{(5)}=
\begin{pmatrix}
1&1&3&4&4\\
1&2&2&2&5\\
3&2&3&2&5\\
4&2&2&4&4\\
4&5&5&4&5
\end{pmatrix}.
\]
Their corresponding unnormalized part-size vectors are
\[
(a+b,a,a,a,a),
\quad
(a+b,a+b,a+b,2a,2a),
\quad
(2(a+b),a+b,a+b,a+b,3a),
\quad
(3,2,2,1,1),
\]
respectively. See Figure~\ref{fig:Kab5-extremal}.

For \(\Lambda_1^{(5)}\), colour \(1\) is supported only on the first part,
while each later colour has a designated part meeting all its edges. For
\(\Lambda_2^{(5)}\), each of the first three colours is supported on two
large parts, while colours \(4\) and \(5\) have designated vertex-cover
parts. For \(\Lambda_3^{(5)}\), each of colours \(1,2,3,4\) is supported
on parts of total normalized size \(3(a+b)/(8a+5b)\), while every
colour-\(5\) copy uses at least \(a\) vertices of the fifth part. Finally,
in \(\Lambda_4^{(5)}\), every colour is supported on parts of total
normalized size at most \(5/9\). These observations give all four upper
bounds in Theorem~\ref{thm:Kab-five}.

\subsection{The five-part forcing lemma}

Define
\[
c_4(a,b):=
\begin{cases}
\displaystyle\frac1{4a+b},
& b\le3a,\\[1ex]
\displaystyle\frac2{5a+3b},
& 3a\le b\le5a,\\[1ex]
\displaystyle\frac3{5(a+b)},
& b\ge5a.
\end{cases}
\]

\begin{lemma}\label{lem:Kab-five-forcing}
Let \(\lambda>0\), and put
$A:=(a+b)\lambda,$ $B:=a\lambda$.
Suppose that
\begin{align}
A+4B&\le1,\label{eq:five-force-1}\\
\frac{1-\frac32A}{2}&\ge B,\label{eq:five-force-2}\\
A&\le\frac59,\label{eq:five-force-3}\\
c_4(a,b)(1-B)&\ge\lambda,\label{eq:five-force-4}\\
\frac ba\cdot\frac{4-7A}{3}&\ge2A-1.\label{eq:five-force-5}
\end{align}
Then every strip colouring with at most five parts and at most five colours
on its cross-pairs contains a monochromatic \(K_{a,b}\)-tiling of size
$\lambda n-O_{a,b}(1).$
\end{lemma}

\begin{proof}
Let the normalized part sizes be
$x_1\ge x_2\ge x_3\ge x_4\ge x_5,$
$\sum_{i=1}^5x_i=1.$
As in the proof of Lemma~\ref{lem:Kab-four-forcing},
\eqref{eq:five-force-1} implies that we may assume
$x_1<\frac ba x_2$ and $x_1+x_2<A.$
Indeed, otherwise either the clique on the first part, the pair formed by the
first two parts, or the inequality
\(1\le x_1+4x_2<A+4B\) gives the result.

Since \(x_3\le x_2\le(x_1+x_2)/2\),
$x_4+x_5
=
1-x_1-x_2-x_3
>
1-\frac{3A}{2}.$
Thus \(x_4>B\) by \eqref{eq:five-force-2}. We may further assume
$x_1<\frac ba x_4,$
since otherwise the pair formed by the first and fourth parts gives the
required tiling.

Put
$T:=x_2+x_3+x_4.$
If \(T\ge A\), two of the six cross-pairs among the first four parts have the
same colour. Adjacent pairs yield a same-coloured complete bipartite graph
on three parts, while disjoint pairs yield two compatible tilings on four
parts. Since all first four parts have size at least \(B\) and size ratio
smaller than \(b/a\), either configuration gives the required tiling.
We may therefore assume
$T<A.$

Let
$S:=x_2+x_3+x_4+x_5.$
Since \(x_2\ge(x_3+x_4+x_5)/3\),
\[
S>\frac{4(1-A)}3
\qquad\text{and hence}\qquad
x_5=S-T>\frac{4-7A}{3}.
\]
Furthermore, the inequalities
\(x_1+x_2<A\), \(x_2+x_3+x_4<A\), and \(x_5\le x_4\) imply
$x_2+x_3<2A-\frac23.$
Consequently,
\[
x_1+x_4+x_5
=
1-x_2-x_3
>
\frac53-2A
\ge A
\]
by \eqref{eq:five-force-3}. Thus every three-part set containing the first
part has total normalized size at least \(A\).

The cross-pairs \(12,13,14\) must now have distinct colours; otherwise two
same-coloured adjacent pairs already give the desired tiling. Denote these
colours by \(c_2,c_3,c_4\).

Suppose first that \(x_5\ge B\). If \(x_1\ge(b/a)x_5\), the pair \(15\)
gives the desired tiling. Otherwise all five parts have pairwise size ratio
smaller than \(b/a\). The colour \(c_5\) of \(15\) must be distinct from
\(c_2,c_3,c_4\), since a repeated colour would give a same-coloured
three-part configuration of total size at least \(A\).

Let \(c_0\) be the only possible remaining colour. Every edge \(ij\) among
\(2,3,4,5\) must have colour \(c_0\): any colour \(c_i\) or \(c_j\) gives
two adjacent same-coloured pairs, while a colour \(c_k\) with
\(k\notin\{i,j\}\) gives two disjoint same-coloured pairs. Hence all
cross-pairs among the final four parts have colour \(c_0\). Taking one of
these parts against the union of the other three, and using \(S>A\), gives
the required tiling.

It remains to consider \(x_5<B\). If the six cross-pairs among the first
four parts use at most four colours, then
Lemma~\ref{lem:Kab-four-forcing}, applied on their
\((1-x_5)n\) vertices, gives at least
\[
c_4(a,b)(1-x_5)n-O(1)
>
c_4(a,b)(1-B)n-O(1)
\ge
\lambda n-O(1)
\]
copies, by \eqref{eq:five-force-4}.

We may therefore assume that those six pairs use all five colours. Let
\(c_5\) be the colour of \(15\). If \(c_5\in\{c_2,c_3,c_4\}\), two
same-coloured adjacent pairs again give the required tiling. Hence \(c_5\)
is distinct from \(c_2,c_3,c_4\). Since all five colours occur among the
first four parts, some edge \(ij\), with \(i,j\in\{2,3,4\}\), has colour
\(c_5\). Thus \(15\) and \(ij\) are disjoint same-coloured pairs.

If \(x_1<(b/a)x_5\), both pairs can tile all their vertices up to
\(O_{a,b}(1)\), and the four parts involved have total size at least
\(S>A\). We may therefore assume \(x_1\ge(b/a)x_5\). The two pairs
give at least
\[
\frac{x_5}{a}n
+
\frac{x_i+x_j}{a+b}n
-O(1)
\]
copies. If this were smaller than \(\lambda n-O(1)\), then
\[
x_i+x_j
<
A-\frac{a+b}{a}x_5.
\]
Since \(x_3+x_4\le x_i+x_j\), we would obtain
\[
x_1+x_2
>
1-A+\frac ba x_5
>
1-A+\frac ba\cdot\frac{4-7A}{3}
\ge A
\]
by \eqref{eq:five-force-5}. But \(x_1<(b/a)x_2\), so the pair \(12\)
would then give the desired tiling, a contradiction.
\end{proof}

\begin{proof}[Proof of Theorem~\ref{thm:Kab-five}]
The four templates above give the upper bounds.

For the lower bounds, apply Lemma~\ref{lem:Kab-five-forcing} with
\[
\lambda=
\frac1{5a+b},
\qquad
\frac2{7a+3b},
\qquad
\frac3{8a+5b},
\qquad
\frac5{9(a+b)}
\]
in the four respective ranges.

Conditions \eqref{eq:five-force-1}--\eqref{eq:five-force-3} follow directly
from the stated bounds on \(b/a\). In the first three ranges,
\[
c_4(a,b)(1-B)=\lambda,
\]
while in the fourth range condition \eqref{eq:five-force-4} is equivalent
to \(2b\ge13a\).

Writing \(z=b/a\), then \eqref{eq:five-force-5} is automatic in the
first range and reduces in the remaining ranges to
\[
2z(7-z)\ge3(z-3),
\qquad
z(11-z)\ge3(z-2),
\qquad
z\ge3,
\]
respectively. These inequalities hold throughout the indicated intervals.

Thus every five-strip colouring contains a monochromatic
\(K_{a,b}\)-tiling of the asserted size. The complete-host structural
reduction transfers the result to arbitrary five-colourings of \(K_n\).
\end{proof}

\section{The double star \(D_{2,3}\)}
\label{app:D23}

In this section, we prove Theorem~\ref{thm:D23}. Write
\(D:=D_{2,3}\). Thus \(v(D)=7\), its bipartition classes have sizes \(3\)
and \(4\), and its vertex-cover number is \(2\).

We begin with the extremal construction.

\begin{lemma}\label{lem:D23upper}
There exists a \(3\)-colouring \(\chi\) of \(E(K_n)\) such that
$\nu_D(\chi)\le \frac{7}{78}n+O(1).$
\end{lemma}

\begin{proof}
Let \(V(K_n)=V_1\cup V_2\cup V_3\), where
\[
|V_1|=\left\lfloor\frac{4n}{39}\right\rfloor,\qquad
|V_3|=\left\lfloor\frac{7n}{26}\right\rfloor,\qquad
|V_2|=n-|V_1|-|V_3|=\frac{49n}{78}+O(1).
\]

Colour the edges inside \(V_1\) and \(V_3\), and the edges between
\(V_1\) and \(V_2\), red. Colour the edges between \(V_3\) and
\(V_1\cup V_2\) blue, and colour the edges inside \(V_2\) green; see
Figure~\ref{fig:D23upper}.

The red graph has two components: the clique on \(V_3\), and the graph
on \(V_1\cup V_2\) in which every edge meets \(V_1\). Since \(D\) has
vertex-cover number \(2\), every red copy in \(V_1\cup V_2\) uses at least
two vertices of \(V_1\), while every red copy in \(V_3\) uses seven vertices
there. Hence a red tiling has at most
$\frac{|V_1|}{2}+\frac{|V_3|}{7}
=
\frac{7n}{78}+O(1)$
copies.

The blue graph is complete bipartite with classes \(V_3\) and
\(V_1\cup V_2\). Since the bipartition classes of \(D\) have sizes \(3\)
and \(4\), every blue copy uses at least three vertices of \(V_3\).
Thus a blue tiling has at most
$\frac{|V_3|}{3}
=
\frac{7n}{78}+O(1)$
copies. Finally, green occurs only inside \(V_2\), so a green tiling has at
most
$\frac{|V_2|}{7}
=
\frac{7n}{78}+O(1)$
copies.
\end{proof}

We next prove the matching lower bound for strip colourings. We use the
finite pattern-packing LP from the proof of Theorem~\ref{thm:finite-optimization}. For fixed
part proportions, its integral optimum differs from its fractional optimum by
\(O_D(1)\).

\begin{lemma}[Three-part packing calculation]
\label{lem:D23-threepart-LP}
Let \(A,U,W\) be three strip parts of normalized sizes \(x,u,v\),
respectively. Suppose that
$AU,\ UW$ are blue,
$AW$ is red,
and that the edges inside \(A\) have a third colour. According to the colours
inside \(U\) and \(W\), the normalized fractional red and blue
\(D\)-tiling numbers are given by the following table:
\[
\begin{array}{c|c|c}
\text{colours inside }U,W
&
\tau_{\mathrm{red}}
&
\tau_{\mathrm{blue}}
\\ \hline
\mathrm{R},\mathrm{R}
&
\displaystyle
\min\left\{\frac{u}{7}+\frac{v}{2},\frac17\right\}
&
\displaystyle
\min\left\{\frac{u}{3},\frac{x+v}{3},\frac17\right\}
\\[3ex]
\mathrm{R},\mathrm{B}
&
\displaystyle
\min\left\{\frac{u}{7}+\frac{v}{3},
      \frac{x}{3}+\frac{u}{7},
      \frac17\right\}
&
\displaystyle
\min\left\{\frac{4u+v}{7},
      \frac{u}{3}+\frac{2v}{9},
      \frac{x}{3}+\frac{v}{2},
      \frac17\right\}
\\[3ex]
\mathrm{B},\mathrm{R}
&
\displaystyle
\min\left\{\frac{v}{2},\frac{x+v}{7}\right\}
&
\displaystyle
\min\left\{\frac{u}{2},\frac17\right\}
\\[3ex]
\mathrm{B},\mathrm{B}
&
\displaystyle
\min\left\{\frac{v}{3},\frac{x}{3},\frac{x+v}{7}\right\}
&
\displaystyle
\min\left\{\frac{4u+v}{7},
      \frac{u}{2}+\frac{v}{6},
      \frac17\right\}.
\end{array}
\]
\end{lemma}

\begin{proof}
For a fixed colour, the pattern LP has one capacity constraint for each of
\(A,U,W\). A copy of \(D\) contributes its numbers of vertices in these
three parts. Enumerating the possible locations of the two centres, and then
of their five leaves, gives the finite primal LP.

The coordinatewise-minimal vertices of the corresponding three-variable
dual LP are as follows:
\[
\begin{array}{c|c|c}
\text{colours inside }U,W
&
\text{red dual vertices}
&
\text{blue dual vertices}
\\ \hline
\mathrm{R},\mathrm{R}
&
(0,\tfrac17,\tfrac12),\
(\tfrac17,\tfrac17,\tfrac17)
&
(0,\tfrac13,0),\
(\tfrac13,0,\tfrac13),\
(\tfrac17,\tfrac17,\tfrac17)
\\[1ex]
\mathrm{R},\mathrm{B}
&
(0,\tfrac17,\tfrac13),\
(\tfrac13,\tfrac17,0),\
(\tfrac17,\tfrac17,\tfrac17)
&
(0,\tfrac47,\tfrac17),\
(0,\tfrac13,\tfrac29),\
(\tfrac13,0,\tfrac12),\
(\tfrac17,\tfrac17,\tfrac17)
\\[1ex]
\mathrm{B},\mathrm{R}
&
(0,0,\tfrac12),\
(\tfrac17,0,\tfrac17)
&
(0,\tfrac12,0),\
(\tfrac17,\tfrac17,\tfrac17)
\\[1ex]
\mathrm{B},\mathrm{B}
&
(0,0,\tfrac13),\
(\tfrac13,0,0),\
(\tfrac17,0,\tfrac17)
&
(0,\tfrac47,\tfrac17),\
(0,\tfrac12,\tfrac16),\
(\tfrac17,\tfrac17,\tfrac17).
\end{array}
\]
Here the coordinates are ordered as \(A,U,W\). Taking the scalar product
with \((x,u,v)\), and applying strong duality, gives the table.
\end{proof}

\begin{lemma}\label{lem:D23lower}
Every \(3\)-strip \(3\)-colouring \(\psi\) of \(K_n\) satisfies
$\nu_D(\psi)\ge \frac{7}{78}n-O(1).$
\end{lemma}

\begin{proof}
Put
$\alpha:=\frac{7}{78}.$
Since the number of strip parts and copy types is fixed, it is enough to
prove that the fractional monochromatic tiling number is at least \(\alpha\).

Pad the strip partition by empty parts if necessary, and let its normalized
part sizes be
\[
x\ge y\ge z,
\qquad
x+y+z=1.
\]
Suppose, for a contradiction, that every colour has fractional
\(D\)-tiling number smaller than \(\alpha\).

The monochromatic clique on the largest part gives \(x/7<\alpha\), so
$x<7\alpha.$
The pair formed by the two largest parts is a monochromatic complete
bipartite graph. Its \(D\)-tiling number is at least
$\min\left\{\frac{x+y}{7},\frac{y}{3}\right\}.$
Since \(x+y=1-z\ge2/3\), the first term is larger than \(\alpha\).
Consequently,
$y<3\alpha.$
It follows that
\[
z=1-x-y>1-10\alpha=\frac8{78}
\qquad\text{and}\qquad
x\ge1-2y>1-6\alpha.
\]

Let \(c_0\) be the colour inside the largest part. We claim that \(c_0\)
occurs on no other pair of strip parts, including the two remaining diagonal
pairs.

If \(c_0\) also occurs inside another part, the two monochromatic cliques
have total size at least \(x+z=1-y>7\alpha\). If \(c_0\) occurs between the
largest part and another part, the resulting monochromatic split graph
contains a \(D\)-tiling of normalized size at least
\[
\min\left\{\frac{x}{2},\frac{x+z}{7}\right\}>\alpha.
\]
Finally, if \(c_0\) occurs between the two smaller parts, we may combine a
tiling inside the largest part with a tiling in the resulting complete
bipartite graph. Its normalized size is at least
\[
\frac{x}{7}
+
\min\left\{\frac{y+z}{7},\frac{z}{3}\right\}
>
\alpha,
\]
using \(x>1-6\alpha\) and \(z>1-10\alpha\). This proves the claim.

Thus all five remaining entries of the strip template use only the other two
colours, which we call red and blue. The two cross-pairs incident with the
largest part must have different colours. Otherwise they form a monochromatic
complete bipartite graph between the largest part and the union of the other
two parts; both sides have size greater than \(3\alpha\), so it contains an
\(\alpha n-O(1)\)-sized \(D\)-tiling.

The third cross-pair agrees in colour with one of these two pairs. Relabel the
two smaller parts as \(U,W\) so that
$AU,\ UW$ are blue,
$AW$ is red,
where \(A\) is the largest part. Write the corresponding normalized sizes as
\(x,u,v\). We have
$x<7\alpha,$ $u,v<3\alpha,$ and $x>1-6\alpha.$

The edges inside \(U\) and \(W\) are each red or blue. We consider the four
possibilities in Lemma~\ref{lem:D23-threepart-LP}.

\smallskip

\noindent\textbf{Case 1: the two diagonal colours are
\(\mathrm{R},\mathrm{R}\).}
The blue tiling bound gives \(u<3\alpha\), while the red bound gives
$\frac{u}{7}+\frac{v}{2}<\alpha.$
Therefore
\[
1=x+u+v
<
7\alpha+u+2\alpha-\frac{2u}{7}
<
9\alpha+\frac{15\alpha}{7}
=
1,
\]
a contradiction.

\smallskip

\noindent\textbf{Case 2: the diagonal colours are
\(\mathrm{B},\mathrm{R}\).}
The blue and red bounds imply, respectively,
\(u<2\alpha\) and \(v<2\alpha\). Hence
\[
1=x+u+v<7\alpha+4\alpha=11\alpha<1,
\]
a contradiction.

\smallskip

\noindent\textbf{Case 3: the diagonal colours are
\(\mathrm{B},\mathrm{B}\).}
The red bound gives \(v<3\alpha\). From the blue bound, either
$4u+v<7\alpha$
 or 
$3u+v<6\alpha.$
In the first case,
\[
u+v
=
\frac{4u+v}{4}+\frac{3v}{4}
<
\frac{7\alpha}{4}+\frac{9\alpha}{4}
=
4\alpha.
\]
The same conclusion follows in the second case from
\[
u+v
=
\frac{3u+v}{3}+\frac{2v}{3}
<
2\alpha+2\alpha.
\]
Thus \(1=x+u+v<11\alpha<1\), again a contradiction.

\smallskip

\noindent\textbf{Case 4: the diagonal colours are
\(\mathrm{R},\mathrm{B}\).}
The red bound gives
$3u+7v<21\alpha.$
The blue bound gives either \(4u+v<7\alpha\) or
\(3u+2v<9\alpha\). In the first case,
\[
u+v
=
\frac4{25}(4u+v)+\frac3{25}(3u+7v)
<
\frac{91}{25}\alpha
<
4\alpha.
\]
In the second case,
\[
u+v
=
\frac4{15}(3u+2v)+\frac1{15}(3u+7v)
<
\frac{19}{5}\alpha
<
4\alpha.
\]
Once more, \(1=x+u+v<11\alpha<1\), a contradiction.

All possibilities lead to a contradiction, proving that some colour has
fractional \(D\)-tiling number at least \(\alpha\). Rounding the fixed
pattern LP loses only \(O(1)\) copies.
\end{proof}

\begin{proof}[Proof of Theorem~\ref{thm:D23}]
Lemma~\ref{lem:D23upper} and Lemma~\ref{lem:D23lower} give
$
\beta_{3,D_{2,3}}=\frac{7}{78}.
$
The result now follows from
Theorem~\ref{thm:finite-optimization}.
\end{proof}

\section{Clique tilings in \(t\)-graphs}
\label{subsec:3-clique}

We next prove Proposition~\ref{prop:clique-tiling-upper}. That is, for clique tilings in \(t\)-graphs, with at least
three colours and clique order strictly larger than the uniformity, the
elementary greedy bound is asymptotically sharp up to an additive constant.

\begin{lemma}\label{lem:pair-colouring}
For every \(r\ge3\), there exists a map
$\rho:\binom{[r]}2\to[r]$
such that:
\begin{enumerate}
\item
\(\rho(\{i,j\})\notin\{i,j\}\) for all distinct \(i,j\in[r]\);

\item
for every three distinct \(a,b,c\in[r]\), the values
$\rho(\{a,b\})$, $\rho(\{a,c\})$,
 $\rho(\{b,c\})$
are not all equal.
\end{enumerate}
\end{lemma}

\begin{proof}
If \(r\) is odd, identify \([r]\) with \(\mathbb Z_r\) and define
$\rho(\{i,j\})=\frac{i+j}{2}\pmod r.$
Since \(2\) is invertible in \(\mathbb Z_r\), this is well-defined and
\(\rho(\{i,j\})\notin\{i,j\}\). For a fixed colour \(c\), each vertex
\(a\) has at most one neighbour \(j\) satisfying
\((a+j)/2=c\). Thus every colour class is a matching, proving the second
property.

Now suppose that \(r\) is even. Identify the colour set with
\(\mathbb Z_{r-1}\cup\{\infty\}\). For distinct
\(i,j\in\mathbb Z_{r-1}\), define
$\rho(\{i,j\})=\frac{i+j}{2}\pmod{r-1},$
and define
$\rho(\{\infty,i\})=i+1\pmod{r-1}.$
The first property is immediate. A triangle contained in
\(\mathbb Z_{r-1}\) is not monochromatic by the odd case, while a triangle
\(\{\infty,a,b\}\) is not monochromatic because its two edges incident with
\(\infty\) have the distinct colours \(a+1\) and \(b+1\).
\end{proof}

\begin{lemma}\label{lem:shadow-separating-colouring}
Let \(t\ge3\) and \(r\ge3\). There exists a map \(\gamma\) from the family
of size-\(t\) multisets on \([r]\) to \([r]\) such that:
\begin{enumerate}
\item
\(\gamma(i^t)=i\) for every \(i\in[r]\), where \(i^t\) denotes the
multiset consisting of \(t\) copies of \(i\);

\item
if \(A\) is a size-\((t+1)\) multiset which is not of the form
\(i^{t+1}\), then the values
\[
\gamma(A-x),
\qquad x\in A,
\]
are not all equal, where \(A-x\) means that one occurrence of \(x\) is
deleted.
\end{enumerate}
\end{lemma}

\begin{proof}
Fix the usual order on \([r]\). For a multiset \(M\), let
\[
O(M):=
\{i\in[r]:\text{the multiplicity of \(i\) in \(M\) is odd}\}.
\]

Suppose first that \(t\) is odd. Then \(O(M)\ne\varnothing\) for every
size-\(t\) multiset \(M\), so define
\(\gamma(M):=\min O(M)\). Clearly \(\gamma(i^t)=i\).

Let \(A\) be a nonconstant size-\((t+1)\) multiset. If \(O(A)=\varnothing\),
choose distinct \(x,y\in\operatorname{supp}(A)\). Then
\(\gamma(A-x)=x\ne y=\gamma(A-y)\). If \(O(A)\ne\varnothing\), its
cardinality is even and at least two. Let \(x<y\) be its two smallest
elements. Then
\(\gamma(A-x)=y\ne x=\gamma(A-y)\).

Now suppose that \(t\) is even, and let \(\rho\) be given by
Lemma~\ref{lem:pair-colouring}. Define
\[
\gamma(M)=
\begin{cases}
\min\operatorname{supp}(M),
& O(M)=\varnothing,\\
\rho(O(M)),
& |O(M)|=2,\\
\min O(M),
& |O(M)|\ge4.
\end{cases}
\]
Since \(|O(M)|\) is even, this is well-defined, and
\(\gamma(i^t)=i\).

Let \(A\) be a nonconstant size-\((t+1)\) multiset. Then \(|O(A)|\) is
odd. If \(O(A)=\{u\}\), and put
\(c:=\gamma(A-u)=\min\operatorname{supp}(A-u)\). 
If \(c\ne u\), then
$\gamma(A-c)=\rho(\{u,c\})\ne c=\gamma(A-u).$
If \(c=u\), choose \(y\in\operatorname{supp}(A)\setminus\{u\}\); then
$\gamma(A-y)=\rho(\{u,y\})\ne u=\gamma(A-u).$

If \(|O(A)|=3\), say \(O(A)=\{a,b,c\}\), deleting \(a,b,c\) produces the
three colours
$\rho(\{b,c\})$, 
$\rho(\{a,c\})$, 
$\rho(\{a,b\})$,
which are not all equal by Lemma~\ref{lem:pair-colouring}. Finally, if
\(|O(A)|\ge5\), let \(x<y\) be its two smallest elements. Then
\(\gamma(A-x)=y\ne x=\gamma(A-y)\).
\end{proof}

\begin{proof}[Proof of Proposition~\ref{prop:clique-tiling-upper}]
Let \(R:=R_r(K_p^{(t)})\). Greedily remove monochromatic copies of
\(K_p^{(t)}\) until fewer than \(R\) vertices remain. This produces at least
\((n-R)/p\) vertex-disjoint copies in total, so one colour contains at least
$\frac{n}{rp}-O_{r,t,p}(1)$
of them. This proves the lower bound.

For the upper bound, let \(\gamma\) be given by
Lemma~\ref{lem:shadow-separating-colouring}, and partition
$V(K_n^{(t)})=V_1\cup\cdots\cup V_r$
as evenly as possible. For a \(t\)-edge \(e\), let \(M(e)\) be the multiset
of its part labels, and define
$
\chi(e):=\gamma(M(e)).
$

We claim that no monochromatic \(K_{t+1}^{(t)}\) meets more than one part.
Indeed, let \(S\) be a \((t+1)\)-set meeting at least two parts, and let
\(A(S)\) be the multiset of its part labels. The colours of the
\(t\)-subsets of \(S\) are precisely
$\gamma(A(S)-x),$
as one occurrence \(x\) ranges over \(A(S)\). These colours are not all
equal by Lemma~\ref{lem:shadow-separating-colouring}.

Since \(p\ge t+1\), every set of \(p\) vertices meeting at least two parts
contains a \((t+1)\)-subset meeting at least two parts. Consequently, every
monochromatic \(K_p^{(t)}\) is contained in one \(V_i\). Moreover, all
\(t\)-edges inside \(V_i\) have colour
\(\gamma(i^t)=i\), so every monochromatic \(K_p^{(t)}\)-tiling of colour
\(i\) is contained in \(V_i\). It therefore has at most
$\left\lfloor\frac{|V_i|}{p}\right\rfloor
\le
\frac{n}{rp}+O_{r,p}(1)$
copies. This proves the upper bound.
\end{proof}
\end{document}